\documentclass[leqno,a4paper,11pt]{amsart}
\usepackage{fullpage}

\usepackage[x11names]{xcolor} 

\usepackage[T1]{fontenc}
\usepackage[utf8]{inputenc} 

\renewcommand{\arraystretch}{1.3}

\usepackage{amsfonts,amsmath,amssymb}
\usepackage[fleqn]{mathtools}
\usepackage{float}
\usepackage{bm}
\usepackage{upgreek}
\usepackage{slashed}
\usepackage{accents}
\usepackage{youngtab}

\usepackage[ddmmyyyy]{datetime}

\usepackage[%
backend=bibtex8,
style=numeric,
giveninits=true,
doi=false,
url=false,
isbn=false,
eprint=false,
maxnames=5
]{biblatex}
\AtEveryBibitem{%
	\clearfield{note}%
}
\renewbibmacro{in:}{} 
\bibliography{biblio}
\DeclareFieldFormat[article,inbook,incollection,inproceedings,patent,unpublished]{title}{\mkbibitalic{#1}}
\DeclareFieldFormat[inproceedings,incollection]{booktitle}{\mkbibquote{#1}}
\DeclareFieldFormat[book,thesis]{title}{\mkbibquote{#1}}
\DeclareFieldFormat[article,inbook,incollection,inproceedings,patent,thesis,unpublished]{volume}{\mkbibbold{#1}}
\DeclareFieldFormat{journaltitle}{#1\isdot}

\usepackage{tikz}
\usepackage{graphicx}
\usepackage[all]{xy}
\usepackage{array}
\usepackage{scalerel}
\usepackage{multirow}
\usepackage[shortlabels]{enumitem}
\usepackage{caption}
\usepackage{hyperref}

\usepackage{amsthm}
\theoremstyle{plain}
\newtheorem{thm}{Theorem}[section]
\newtheorem{lem}[thm]{Lemma}
\newtheorem{prop}[thm]{Proposition}
\newtheorem{cor}[thm]{Corollary}
\theoremstyle{definition}
\newtheorem{defn}[thm]{Definition}

\newtheorem{exa}[thm]{Example}

\theoremstyle{remark}
\newtheorem{rem}[thm]{Remark}

\numberwithin{equation}{section}
\renewcommand{\d} {\mathrm{d}}

\renewcommand{\div} {\mathrm{div}}

\newcommand{\wh} [1]{\widehat{#1}}
\newcommand{\wt} [1]{\widetilde{#1}}

\newcommand{\ul} [1]{\underline{#1}}

\newcommand{\mbf} [1]{\mathbf{#1}}
\newcommand{\mbb} [1]{\mathbb{#1}}
\newcommand{\mc} [1]{\mathcal{#1}}
\newcommand{\mf} [1]{\mathfrak{#1}}

\renewcommand{\i} {\mathrm{i}}
\newcommand{\e} {\mathrm{e}}

\renewcommand{\bigwedge}{\scaleobj{1.2}{\wedge}}
\renewcommand{\bigotimes}{\scaleobj{1.2}{\otimes}}
\renewcommand{\bigodot}{\scaleobj{1.2}{\odot}}

\newcommand{\hook}{\makebox[7pt]{\rule{6pt}{.3pt}\rule{.3pt}{5pt}}\,}

\newcommand{\Ann} {\mathrm{Ann}}

\newcommand{\gr} {\mathrm{gr}}

\newcommand{\su} {\mathfrak{su}}

\newcommand{\so} {\mathfrak{so}}
\newcommand{\co} {\mathfrak{co}}

\newcommand{\g} {\mathfrak{g}}
\newcommand{\h} {\mathfrak{h}}
\newcommand{\p} {\mathfrak{p}}

\newcommand{\R} {\mathbf{R}}
\newcommand{\C} {\mathbf{C}}

\newcommand{\U}{\mathbf{U}}
\newcommand{\SU}{\mathbf{SU}}

\renewcommand{\O}{\mathbf{O}}
\newcommand{\SO}{\mathbf{SO}}

\newcommand{\CO}{\mathbf{CO}}
\newcommand{\Sim}{\mathbf{Sim}}

\newcommand{\V}{{\mathbb V}}

\newcommand{\K}{{\mathbb K}}

\begin{document}

	\title
	{Optical geometries}

	\begin{abstract}
		We study the notion of optical geometry, defined to be a Lorentzian manifold equipped with a null line distribution, from the perspective of intrinsic torsion. This is an instance of a non-integrable version of holonomy reduction in Lorentzian geometry. These generate   congruences of null curves, which play an important r\^{o}le in general relativity. Conformal properties of these are investigated. We also extend this concept to generalised optical geometries as introduced by Robinson and Trautman.
	\end{abstract}
	
	\date{\today\ at \xxivtime}
	
	\author{Anna Fino}\address{Dipartimento di Matematica ``G. Peano'', Universit\`{a} degli studi di Torino,  Via Carlo Alberto 10, 
		10123 Torino, Italy.  \& Department of Mathematics and Statistics, Florida International University, Miami, FL 33199, United States}	\email{annamaria.fino@unito.it, afino@fiu.edu}
	\author{Thomas Leistner}\address{School of Computer and Mathematical Sciences, University of Adelaide, SA 5005, Australia}\email{thomas.leistner@adelaide.edu.au}
	\author{Arman Taghavi-Chabert}\address{Institute of Physics, {\L}\'{o}d\'{z} University of Technology, W\'{o}lcza\'{n}ska 217/221, 90-005 {\L}\'{o}d\'{z}, Poland}\email{arman.taghavi-chabert@p.lodz.pl}
	
	\thanks{AF was supported by GNSAGA of INdAM, by PRIN 2017 \lq \lq Real and Complex Manifolds: Topology, Geometry and Holomorphic Dynamics" and by the grant \# 944448 from the Simons Foundation.
		TL was supported by
		the Australian Research
		Council (Discovery Program DP190102360). ATC \& TL declare that this work was partially supported by the grant 346300 for IMPAN from the Simons Foundation and the matching 2015-2019 Polish MNiSW fund. The research of ATC leading to these results has received funding from the Norwegian Financial Mechanism 2014-2021 UMO-2020/37/K/ST1/02788. He has received funding from the GA\v{C}R (Czech Science Foundation) grant 20-11473S. He was also supported by a long-term faculty development grant from the American University of Beirut for his visit to IMPAN, Warsaw, in the summer 2018, where parts of this research was conducted.}
	\subjclass[2010]{Primary 53C50, 53C10; Secondary 53B30, 53C18}
	\keywords{Lorentzian manifolds, almost Robinson structures, G-structure, intrinsic torsion, congruences of null geodesics, conformal geometry, almost CR structures}
	
	\maketitle
	\tableofcontents

	\section{Introduction}
	One of the most natural geometric structures that one can study on a Lorentzian manifold $(\mc{M},g)$ of dimension $n+2$ is a null line distribution $K$, say. This is what we shall term an \emph{optical structure}, and it is equivalent to a reduction of the structure group of the frame bundle to $\Sim(n)$, the stabiliser of a null direction in $\R^{n+1,1}$.
	The aim of this article will be to give a comprehensive study of $\Sim(n)$-structures, including their \emph{intrinsic torsion}, which measures the extent of their non-integrability. Integrability here is equivalent to $K$ being parallel with respect to the Levi-Civita connection.
	
	The purpose of this article is two-fold. First it provides a formal approach to well-known concepts of mathematical relativity in relation with \emph{congruences of null curves}, i.e., the foliation tangent to $K$. Second, it lays out foundational material for further works in which the structure group of the frame bundle is reduced further to a subgroup of $\Sim(n)$ when $n=2m$. In particular, it should be viewed as an introductory article to the follow-up paper \cite{Fino2023} on almost Robinson manifolds, 
	where the structure group is reduced to $(\R^*\times \U(m))\ltimes (\R^n)$. The study of such structures in arbitrary even dimension goes back to Hughston and Mason \cite{Hughston1988}, and  the notion of an (almost) Robinson structure was introduced by Nurowski and Trautman in \cite{Nurowski2002,Trautman2002,Trautman2002a}. In fact, this reduction is vacuous in dimension four, where an almost Robinson structure is identical to an optical structure (for further comments on terminology, see Remark \ref{term-remark}). Here, the ramifications into complex geometry and CR geometry would prove crucial in the formulation of twistor theory on the one hand \cite{Penrose1967,Penrose1986}, and particularly powerful in the understanding of solutions to the Einstein field equations on the other hand \cite{Robinson1986,nurowski93-phd,Lewandowski1990,Nurowski1997}.

	In mathematical relativity, congruences of null geodesics are fundamental in the study of the Einstein field equations, especially in dimension four, but also in higher dimensions, in the null alignment formalism of \cite{Coley2003,Milson2005,Durkee2010}. The invariant properties of a congruence of null \emph{geodesics} are encoded in the so-called \emph{optical scalars} in dimension four, and \emph{optical matrices} in higher dimensions, namely the expansion, twist and shear of the congruence. These correspond to $\Sim(n)$-invariant irreducible pieces of the intrinsic torsion of the $\Sim(n)$-structure. For instance, the existence of \emph{non-shearing} congruences is intimately connected to solutions to field equations such as the Einstein equation or the Maxwell equation in vacuum via central results such as the Goldberg-Sachs theorem \cite{Goldberg1962,Goldberg2009,GoverHillNurowski11} and the Robinson theorem \cite{Robinson1961} --- see also \cite{nurowski96}.
	
	Earlier investigations on the classification of intrinsic torsions on pseudo-Riemannian manifolds can be found in \cite{Taghavi-Chabert2016,Taghavi-Chabert2017a,Papadopoulos2019,FigueroaOFarrill2020} -- see also in \cite{Petit2019} for the construction of connections adapted to totally null distributions.
	
	The structure of the paper is as follows. In Section \ref{sec:G-structures}, we give a general framework for the investigation of reductions of the structure group in Lorentzian geometry. In Section \ref{sec:algebra}, we obtain a description of the intrinsic torsion of a $\Sim(n)$-structure in terms of irreducibles. This is applied to the geometric setting in Section \ref{sec:geometry} where we introduce optical geometries and examine its relation to congruences of null curves and their leaf spaces. In particular, a non-integrable optical geometry with congruence of null geodesics can therefore fall into eight distinct classes according to whether the expansion, twist or shear vanishes or not. In Section \ref{sec:conf-opt-str}, we substitute a conformal structure for the Lorentzian structure, and defines an optical structure in that setting. Section \ref{sec:four_dim} briefly describes the situation in dimension four. Finally, we revisit the notion of (generalised) optical geometry first considered in \cite{Robinson1985} in Section \ref{sec:gen_opt}: these consist in weakening the conformal structure to a certain equivalence class of Lorentzian metrics on a smooth manifold. In this case, the structure group is a Lie subgroup of $\mathbf{GL}(n+2,\R)$ stabilising a filtration, and a conformal structure on its associated screen bundle.
	
	Throughout, we provide a number of examples to illustrate the algebraic conditions that the intrinsic torsion of an optical geometry can satisfy. The reference \cite{Stephani2003} contains a plethora of such solutions in dimension four where the properties of the intrinsic torsion can easily be inferred in terms of congruences of null geodesics. In higher dimension, the review article \cite{Ortaggio2013} is a convenient source.

	\section{$H$-structures and their intrinsic torsion}\label{sec:G-structures}
	Let $(\mc{M},g)$ be a semi-Riemannian manifold of dimension $m$ and signature $(s,t)$ with the convention that $t$ is the number of negative eigenvalues of the metric $g$, and $s$ is  the number of positive ones. Let $G:=\SO^0(s,t)$ be the connected component of the group of orthogonal transformations of $\R^{s,t}$, and $H\subset G$ a closed subgroup
	of $G$. Then an {\em $H$-structure on $(\mc{M},g)$} is a reduction of the $G$-bundle of orthonormal frames $\mc F^G$ to an $H$-bundle $\mc F^H$. We do {\em not} assume that the Levi-Civita connection of $g$ reduces to $H$, and if it does not, an $H$-structure is called {\em non-integrable}. 
	
	We now recall the notion of {\em intrinsic torsion} of an $H$-structure. This requires some algebraic preliminaries in which we use abstract index notation. To this end we set $\V=\R^{s,t}$ and denote by $\h$ and $\g=\so(s,t)=\so(\V)$ the respective Lie algebras of $H$ and $G$, all of which are $H$-modules in a canonical way. Elements of $\V$ will be adorned with upper lower-case Roman indices starting from the beginning of the alphabet, e.g.\ $v^a, w^b \in \V$, and elements of the dual $\V^*$ by lower indices, e.g.\ $\alpha_{a}, \beta_{b} \in \V^*$. The Minkowski inner product on $\V$ will be denoted $g_{a b}$, with inverse $g^{a b}$, and will be used to lower and raise indices. 
	
	We introduce the $G$-module homomorphism 
	\begin{align*} \delta_\g: \V^*\otimes \g \ni A_{ab}^{\ \ c}\mapsto A_{[ab]}^{\ \ \ \ c}\in  \bigwedge^2\V^* \otimes \V,\end{align*} 
	where $ A_{[ab]}^{\ \ \ c}=\tfrac{1}{2}( A_{ab}^{ \ \ c}- A_{ba}^{ \ \ c})$ denotes the skew-symmetrisation of $A_{ab}^{ \ \ c}$ in the covariant components.  Using that  $\V^*\otimes \g \simeq \V^*\otimes \bigwedge^2\V^*$, a direct computation will reveal that the first prolongation
	$$\g^{(1)} := \{ A_{a b}{}^{c} \in \V^*\otimes \g : A_{[a b]}{}^{c} = 0 \}$$
	of $\g$ is zero (see e.g.\ \cite{Salamon1989}). Thus, $\delta_\g$ is an isomorphism, and in particular, the restriction $\delta_\h$ of $\delta_\g$ to $\V^*\otimes \h$ is an isomorphism onto its image $\mathrm{Im}(\delta_\h)$. We then obtain the following commutative
	diagram  of $H$-modules:
	\begin{equation}\label{isom}
		\begin{array}{rcl}
			\delta_{\g}\ :\ \V^*\otimes \g &\simeq& \bigwedge^2\V^*\otimes \V\\[2mm]
			\pi\ \downarrow\quad&&\quad\downarrow \ \hat \pi \\[2mm]
			\mbb{G}:=\V^*\otimes (\g/\h)
			&\simeq &
			\hat{\mbb{G}} :=\left( \bigwedge^2 \V^*\otimes \V\right)/_{\mathrm{Im}(\delta_\h) }\\[ 2mm]
			\alpha\otimes (A+ \h) &\mapsto & \delta_\g(\alpha\otimes A)+ \mathrm{Im}(\delta_\h),\end{array}\end{equation}
	where $\pi$ and $\hat \pi$ denote the canonical projections, and $A\in \g$.
	
	Given an $H$-structure $\mc F^H$, we can now associate to every $H$-module $\mbb{A}$  the corresponding vector bundle on $\mc{M}$ with fibre $\mbb{A}$, which we denote by $\mc F^H(\mbb{A})=\mc F^H\times_H \mbb{A}$, e.g., $T \mc{M}=\mc F^G(\V)=\mc F^H(\V)$, etc. Moreover, to every connection of $\mc F^H$ we can associate the torsion of the corresponding covariant derivative $\nabla$, which we denote by $T^\nabla\in \Gamma(\bigwedge^2 T^* \mc{M}\otimes T \mc{M})$. Note that since $H\subset G$, $\nabla$ is compatible with the metric $g$. The {\em intrinsic torsion of the $H$-structure} is then defined as 
	\begin{align*} T^H:=\hat \pi(T^\nabla)\ \in\  \Gamma(\mc F^H(\mbb{G})),\end{align*} 
	where $\hat \pi$ is the vector bundle projection induced from the projection $\hat \pi$ in the diagram~(\ref{isom}). This definition is independent of the chosen $\h$-valued connection $\nabla$: if $\nabla^\prime $ is another $\h$-valued connection, then $\nabla-\nabla^\prime$ is a section of $T^* \mc{M}\otimes \mc F^H(\h)$ and   $T^\nabla-T^{\nabla^\prime}=\delta_\h(\nabla-\nabla^\prime) $, which projects to zero under $\hat{\pi}$.
	
	Although the Levi-Civita connection $\nabla^g$ is not an $\h$-valued connection for non-integrable $H$-structures, it can be used to describe the intrinsic torsion using the geometric version of the diagram~(\ref{isom})
	\begin{equation}\label{isom1}
		\begin{array}{rcl}
			\delta_{\g}\ :T^* \mc{M} \otimes \mc F^H( \g) &\simeq& \bigwedge^2T^* \mc{M}\otimes T^* \mc{M}\\[ 2mm]
			\pi\ \downarrow\quad&&\quad\downarrow \ \hat \pi \\[ 2mm]
			T^* \mc{M}\otimes \mc F^H(\g/\h)
			&\simeq &
			\mc F^H(\hat{\mbb{G}})
			\, .
	\end{array}\end{equation}
	Indeed, for a given $\h$-connection $\nabla$, we have that $C:=\nabla-\nabla^g$ is a section of $T^* \mc{M}\otimes \mc F^H(\g)
	\subset 
	T^* \mc{M}\otimes T^* \mc{M}\otimes T \mc{M}$, which is isomorphic via the metric $g$ to $
	T^* \mc{M}\otimes\bigwedge^2  T^* \mc{M}$. Then, as $\nabla^g$ is torsion free,  $\delta_\g(C)$ is just the torsion of $\nabla$ and projects under $\hat \pi$ to the intrinsic torsion $T^H$. Going down the left path in the diagram~(\ref{isom1}), the intrinsic torsion $T^H$ can be identified with $\pi(C)$.  
	In fact for $p\in \mc{M}$ we have that
	\begin{align*} 
		C|_p=g^{b c}g(\nabla \sigma_a- \nabla^g \sigma_a,\sigma_c)\sigma^a\otimes \sigma_b |_p,\end{align*} 
	where $(\sigma_1,\ldots,\sigma_m)$ is a local section of $\mc F^H$ and $g^{a b}$ is the inverse matrix of $g_{a b}=g(\sigma_a,\sigma_b).$
	If $\nabla$ is $\h$-connection, then for $v\in T_p \mc{M}$, the matrix $(g^{b c}g (\nabla_v\sigma_a, \sigma_c)|_p)_{a,b=1}^m$  is in $\h$. On the other hand,  $(g^{b c}g (\nabla^g_v\sigma_a, \sigma_c)|_p)_{a,b=1}^m$ is a matrix in $\g$, and hence
	the intrinsic torsion can be identified with $T^H\in T^* \mc{M}\otimes \mc F^H(\g/\h),$ defined by
	\begin{align}\label{eq-intors-gen}
		T^H(v)|_p=
		\left[(\sigma_1, \ldots, \sigma_m),
		\pi \left((g^{b c}g (\nabla^g_v\sigma_a, \sigma_c)|_p)_{a,b=1}^m\right)\right]\in \mc F^H(\g/\h)=\mc F^H\times_H\g/\h \, ,
	\end{align} 
	where $v\in T_pM$, and  $\pi:\g\to\g/\h$ is the canonical projection.

	This can be made more explicit in the special case when $g$ is Riemannian, i.e. when $\g=\so(m)$.  In this case $\h$ has an $H$-invariant complement\footnote{This also holds in the more general situation when $g$ is not positive definite but $\h$ is reductive} $\h^\perp$ within $\g$, given by the orthogonal complement of $\h$ with respect to the negative definite Killing form of $\g$. Then we can identify $\mbb{G}$ with $\V^*\otimes \h^\perp$. In this case one can define a distinguished  $\h$-connection $\nabla$ by projecting the Levi-Civita connection onto $\h$ with respect to the decomposition $\g=\h\oplus\h^\perp$. This connection is usually called the {\em characteristic connection of the $H$-structure}.
	
	If in addition $H$ is given by the stabiliser of a tensor $\xi$, for example, a stable $3$-form in the case of $\mbf{G}_2$-structures, or a K\"ahler form in the case of $\U(n)$-structures, one can show \cite{Gray1980,Fernandez1982, Salamon1989} that $\nabla \xi$ can be identified with a section of $\mc F^H(\mbb{G})\simeq T^* \mc{M}\otimes \mc F^H(\h^\perp)$. The decomposition of $\V^* \otimes \h^\perp$ into irreducible $H$-modules then gives rise to the Gray-Hervella and Gray-Fern\'andez classifications of almost Hermitian structures and non-integrable $\mbf{G}_2$-structures, respectively. In the present article we shall follow the same approach for $(n+2)$-dimensional Lorentzian manifolds, i.e., with $\h\subset \g=\so(n+1,1)$ for $\h$ being a Lorentzian holonomy algebra acting indecomposably on $\R^{n+1,1}$. The crucial difference however is that $\h$ does not have an $H$-invariant complement in $\g$, so we will have to work with $\g/\h$.

	\section{Algebraic description}\label{sec:algebra}
	\subsection{Linear algebra}\label{algsec}
	Let $\V:=\R^{n+1,1}$ be the Minkowski space of dimension $n+2$. An \emph{optical structure} on $\V$ is a one-dimensional vector subspace $\K$ of $\V$ that is null with respect to the Minkowski inner product $\langle \cdot, \cdot \rangle$ on $\V$, i.e., $\langle k , k \rangle = 0$ for any vector $k$ in $\V$. This subspace $\K$ is contained in its orthogonal complement $\K^\perp$, i.e.\
	\begin{align}\label{eq:K-filtration}
		\K \subset \K^\perp \subset \V.
	\end{align}
	The \emph{screen space} of $\K$ is the quotient space $\mbb{H}_\K := \K^\perp / \K$. It inherits a non-degenerate symmetric bilinear form given by
	\begin{align*}
		\langle v + \K , w + \K \rangle_{\mbb{H}_\K} & := \langle v,w \rangle_\V  \, , & \mbox{for any $v, w \in \K^\perp$.}
	\end{align*}
	Since $\K$ is null,  this is independent of the choice of representatives in $\K^\perp/\K$. 
	
	We may choose a null line $\mbb{L}$ in $\V$ that is dual to $\K$ so that $\V = \K \oplus \left( \K^\perp \cap \mbb{L}^\perp \right) \oplus \mbb{L}$, and introduce a semi-null basis $(e_0, e_1,\ldots , e_n, e_{n+1})$ such that $\K = \mathrm{span}(e_{n+1})$, $\mbb{L} = \mathrm{span}(e_{0})$, $\K^\perp \cap \mbb{L}^\perp = \mathrm{span}(e_1, \ldots , e_n)$, and
	\begin{align*} 
		\left( \langle e_a, e_b\rangle_{n+1,1} \right)_{a,b=0}^{n+1} =
		\begin{pmatrix}
			0 & 0 & 1\\
			0 & \mbf{1}_n & 0\\
			1 & 0 & 0
		\end{pmatrix}
		\, .
	\end{align*}
	In particular, $\langle e_0,e_{n+1}\rangle =1$, $e_0$ and $e_{n+1}$ are null, i.e., $\langle e_0 , e_0 \rangle = \langle e_{n+1} , e_{n+1} \rangle = 0$, and $(e_1 , \ldots , e_{n})$ are orthonormal, i.e., $\langle e_i , e_j \rangle = \delta_{i j}$ for $i,j=1, \ldots, n$. By a slight abuse of notation we denote by
	$\O(n+1,1)$ all  $(n+2)\times(n+2)$ matrices $A$ such that $\langle A v , A w \rangle = \langle v , w \rangle$,  for any $v,w$ in $\V$, by $\SO(n+1,1)$ the set of all matrices in $\O(n+1,1)$ that have determinant one, and by $\SO^0(n+1,1)$ its connected component.
	
	The stabiliser of the line $\K$ in $\V$, and thus of the filtration \eqref{eq:K-filtration}, in $\SO(n+1,1)$ is the maximal parabolic subgroup given by
	\begin{eqnarray*}
	\Sim(n) &=& \CO(n)\ltimes (\R^n)^* \ =\  (\R^* \times \SO(n))\ltimes (\R^n)^*
	\\
	& = &
	\left\{(a,A,\xi):=
	\begin{pmatrix} a & 0 & 0  \\
		-a A\xi^\top &A & 0 \\
		-\tfrac{a}{2} \xi\xi^\top & \xi &a^{-1}
	\end{pmatrix} 
	\mid \begin{array}{l} a \in \R^* \, ,\\
		A \in \SO(n) \, ,\\
		\xi\in (\R^n)^*
	\end{array}
	\right\} \, .
\end{eqnarray*}
	Here $\CO (n)=\R^*\times \SO(n)$ is the special conformal group of Euclidean space $\R^n$, where we use the identification of $(a,A,0)\in \R^*\times \SO(n)\subset \Sim(n)$ with $a A \in \CO(n)$.
	The notation $\Sim(n)$ comes from the fact that it acts as the group of similarity transformations of $\R^n$, given by homotheties (excluding reflections) and translations.
	
	One can alternatively consider the stabiliser of the null \emph{ray} $\R^+ \cdot e_{n+1}$ in $\SO(n+1,1)$: this is simply the connected component $\Sim^0(n)$ of $\Sim(n)$, i.e., $\Sim^0(n) = (\R^+\times \SO(n))\ltimes (\R^n)^*$. Its Levi part is the connected component $\CO^0 (n)=\R^+\times \SO(n)$.
	
	Note that the action of $\Sim(n)$ on the screen space $\mbb{H}_\K=\K^\perp/\K$ is given by $(a,A,\xi)\cdot [x]=[Ax]\in \mbb{H}_\K$ where $(a,A,\xi)\in \Sim(n)$  and $[x]\in \mbb{H}_\K$. In particular, $\Sim(n)$, and thus $\Sim^0(n)$, preserve the induced orientation on $\mbb{H}_\K$.
	
	Within $\Sim(n)$ there is the special Euclidean group $\mathbf{Euc}^0(n)=\SO(n)\ltimes (\R^n)^*$, which is the stabiliser of the null vector $\mbf{e}_0$ in its representation on $\R^{n+1,1}$.
	
	An interesting property about  $\SO^0(n+1,1)$ is that it has no proper connected subgroups that act {\em irreducibly} on $\R^{n+1,1}$, i.e., without invariant subspace, see for example \cite{olmos-discala01}. Hence,  proper subgroups of $\SO^0(n+1,1)$ that act indecomposably, i.e., without {\em nondegenerate} invariant subspace, are contained in $\Sim^0(n)$. For future use, we set
	\begin{align*}
		P & := \Sim^0(n) \, , & P_0 & := \CO^0(n) \, .
	\end{align*}

	The Lie algebra
	$\g=\so(n+1,1)$ is $|1|$-graded
	\begin{align*} \g=\g_{-1}\oplus\g_0\oplus\g_1
	\end{align*} 
	with 
	\begin{eqnarray*} \g_0&=&
		\left\{\begin{pmatrix} r&0&0  \\ 0 &Z &0\\ 0&0& -r\end{pmatrix}\mid \begin{array}{l} r \in \R,
			Z\in \so(n)
		\end{array}\right\}=\co(n)=\R\oplus \so(n),
		\\
		\g_{1}&=&\left\{\begin{pmatrix}0& 0 &0  \\ -\xi^\top &0 & 0\\ 0&\xi &0\end{pmatrix}\mid \begin{array}{l}  \xi\in (\R^n)^*\end{array}\right\}
		=(\R^n)^*
		\\
		\g_{-1}&=&\left\{\begin{pmatrix}0& -x^\top &0  \\ 0 &0 & x \\ 0& 0 &0\end{pmatrix}\mid \begin{array}{l}  x\in \R^n\end{array}\right\}
		=\R^n,
	\end{eqnarray*}
	i.e., the Lie bracket in $\so(n+1,1)$ satisfies $[\g_i,\g_j]=\g_{i+j}$ for $i,j=-1,0,1$, with the convention that $\g_i = \{ 0 \}$ for all $|i|>1$.
	
	Now, setting $\g^1 := \g_1$, $\g^0 := \g_0 \oplus \g_1$ and $\g^{-1} :=  \g_{-1} \oplus \g_0 \oplus \g_1$, we obtain a filtration
	\begin{align}\label{eq-sim-filtration-g}
		\{ 0 \} =: \g^2 \subset \g^1 \subset \g^0 \subset \g^{-1} = \g \, .
	\end{align}
	of $P$-modules. The adjoint representation of $P_0$ acts
	\begin{itemize}
		\item on $\g_0=\co(n)$ via the adjoint representation, 
		\item on $\g_{-1}=\R^n$ via the standard representation 
		\item on $\g_{1}=(\R^n)^*$ via the dual representation,
	\end{itemize}
	as can be gleaned from the following computation:
	\begin{align*} 
		\begin{pmatrix}
			a & 0 &0 \\
			0 & A & 0 \\
			0 & 0 & a^{-1}
		\end{pmatrix}
		\begin{pmatrix}
			r & -x^\top & 0 \\
			-\xi^\top & Z & x\\
			0 & \xi & -r
		\end{pmatrix}
		\begin{pmatrix}
			a^{-1} & 0 & 0 \\
			0 & A^\top & 0 \\
			0 & 0 &a
		\end{pmatrix}
		=
		\begin{pmatrix}
			r & -a x^\top A^\top & 0 \\
			-a^{-1} A \xi^\top & A Z A^\top & a A x \\
			0 & a^{-1}  \xi A^\top & -r
		\end{pmatrix} \, , 
	\end{align*} 
	where $a \in \R^*$, $A \in \SO(n)$, $r \in \R$, $Z \in \so(n)$, $\xi \in (\R^n)^*$ and $x \in \R^n$.
	
	The Lie algebra of $P$ is given by
	\begin{eqnarray*}
		\g^0 = \p=\mf{sim}(n)&=&\g_0\ltimes \g_{1}\ =\ \co(n)\ltimes (\R^n)^*\ =\ (\R\oplus\so(n))\ltimes (\R^n)^*
		\\
		&=&\left\{\begin{pmatrix} r& 0 &0  \\ -\xi^\top & Z & 0\\ 0& \xi & -r\end{pmatrix}\mid \begin{array}{l} r \in \R \, ,\\
			Z\in \so(n) \, ,\\
			\xi\in (\R^n)^*\end{array}\right\} \, .
	\end{eqnarray*}
	The abelian part $P_+$ of $P$ is given by exponentiation of $\g_1$. Clearly the quotient $\g/\p$ is isomorphic as a $P_0$-module to $\g_{-1}$, which is {\em not } a $P$-module.

	Similarly, $\V$ splits into a direct sum of $P_0$-modules
	\begin{align}\label{eq-sim-grading}
		\V & = \V_{-1} \oplus \V_0 \oplus \V_{1} \, ,
	\end{align}
	with 
	$\V_{1}=\K$, $\V_{-1}=\mbb{L}$ and $\V_{0}=\left( \K^\perp \cap \mbb{L}^\perp \right)$  and such that
	\begin{align*} \g_i\cdot \V_j\subset \V_{i+j}.\end{align*} 
	The vector space $\V$ admits a filtration of $P$-modules:
	\begin{align}\label{eq-sim-filtration}
		\{ 0 \} =: \V^2 \subset \V^1 \subset \V^0 \subset \V^{-1} := \V \, ,
	\end{align}
	where $\V^1=\V_1=\K$, $\V^0=\V_1\oplus\V_0=\K^\perp$. The \emph{associated graded vector space $\gr(\V)$} of the filtration \eqref{eq-sim-filtration} is given by
	\begin{align*}
		\gr(\V) & := \gr_{-1} (\V) \oplus \gr_0 (\V) \oplus \gr_1 (\V), & &  \mbox{where} & \gr_i (\V) & := \V^i / \V^{i+1} \, .
	\end{align*}
	Each $\mathrm{gr}_i(\V)$ is a $P$-module, which, as a $P_0$-module, is isomorphic to $\V_i$.
	
	The Minkowski inner product isomorphism $\V^* \cong \V$ induces a filtration
	\begin{align}\label{eq-sim-filtration-dual}
		\{ 0 \} =: (\V^*)^2 \subset (\V^*)^1 \subset (\V^*)^0 \subset (\V^*)^{-1} := \V^*, 
	\end{align}
	where $(\V^*)^i=(\V^i)^*$ and a splitting, 
	\begin{align}\label{eq-sim-grading-dual}
		\V^* & = \V^*_{-1} \oplus \V^*_0 \oplus \V^*_{1} \, ,
	\end{align}
	on the dual vector space $\V^*$.

	At this stage, we recall our conventions for index types already introduced at the beginning of this section, and these will be used either \emph{concretely}, as numerical indices, or \emph{abstractly}, that is, we shall view them as markers of membership of some vector space. Thus:
	\begin{itemize}
		\item Upper minuscule Roman letters starting from the beginning of the alphabet, $a, b, c, \ldots$ will denote elements of $\V$, and lower ones elements of $\V^*$, e.g. $v^a, w^b \in \V$ and $\alpha_{a}, \beta_{b} \in \V^*$. This notation is extended to tensor products in the obvious way, with symmetrisation and skew-symmetrisation denoted by round and square brackets respectively, e.g.\ $\sigma_{(a b)} = \frac{1}{2}\left( \sigma_{a b} + \sigma_{b a} \right) \in \bigodot^2 \V^*$ and $\tau_{[a b]} = \frac{1}{2}\left( \tau_{a b} - \tau_{b a} \right) \in \bigwedge^2 \V^*$. In particular, the inner product on $\V$ will be denoted by $g_{a b}$, and its inverse $g^{a b}$, and will be used to lower and raise indices. The tracefree part of a symmetric tensor will be adorned with a ring, e.g.\ $\left(T_{ab} \right)_{\circ}$ (or $T_{(ab)_\circ}$) satisfies $\left(T_{ab} \right)_{\circ} g^{a b} = 0$.
		\item Once a splitting is chosen as in \eqref{eq-sim-grading} and \eqref{eq-sim-grading-dual}, Upper minuscule Roman letters starting from the middle of the alphabet, $i, j, k, \ldots$ will denote elements of $\V_0$, and lower ones elements of $\V^*_0$, e.g. $v^i, w^j, \ldots \in \V_0$ and $\alpha_{i}, \beta_{j} \in \V^*_0$. This notation will be extended to elements of $\K^\perp/\K$ and their dual. The screen space inner product on $\K^\perp / \K$ yields an inner product on $\V_0$ denoted $h_{i j}$, and with inverse $h^{i j}$. Indices can be raised or lowered accordingly. The conventions for symmetrisation, skew-symmetrisation and tracefreen part will be the same as above, e.g.\ $\tau_{[i j]} = \frac{1}{2}\left( \tau_{i j} - \tau_{j i} \right) \in \bigwedge^2 \V^*$.
	\end{itemize}
Index types can be mixed, e.g. $T_{a b}{}^{i}{}_{j} \in \bigotimes^2 \V^* \otimes \V_{0} \otimes \V_{0}^*$, and so on. We can also view $T_{a b}{}^{i}{}_{j}$ as the components of some tensor $T \in \bigotimes^2 \V^* \otimes \V_{0} \otimes \V_{0}^*$.
	
	Once a splitting is fixed, we can choose elements $k^a \in \V_1 = \K$ and $\ell^a \in \V_{-1} = \mathbb{L}$ such that $g_{a b} k^a \ell^b =1$, and introduce a surjective map
	$$\delta_a^i :  \V \rightarrow \V_0 \, ,$$
	so that $k^a \delta_a^i = \ell^a \delta_a^i = 0$. Raising and lowering indices yield a surjective map $\delta_i^a :  \V^* \rightarrow (\V_0)^*$. In particular, the screen space inner product can be expressed as
	$h_{ij} = g_{a b} \delta^a_i \delta^b_j$. We shall also use these dually as injective maps
	\begin{align*}
		\delta_a^i & :  (\V_0)^* \rightarrow \V^* \, , &  \delta_i^a & :  \V_0 \rightarrow \V \, .
	\end{align*}
	Thus, in this notation, any element $v^a$ of $\V$ that satisfies $g_{a b} v^a k^b = g_{a b} v^a \ell^b = 0$ can be expressed uniquely as $v^a = v^i \delta^a_i$ for some $v^i \in \V_0$. We shall refer to the tuple $(\ell^a, \delta^a_i,k^a)$ and its dual $(\kappa_a,\delta_a^i, \lambda_a)$ as \emph{splitting operators}.

	Any change of splitting that keeps $k^a$ fixed is effected by the action of an element \[\phi_a{}^b = \delta_a^b + \phi_i \delta^i_a k^b - k_a \phi^i \delta_i^b - \frac{1}{2} \phi^i \phi_i k_a k^b\] of $P_+$ on $\V$. In particular, the elements $k^a$, $\delta^a_i$ and $\ell^a$ transform as
	\begin{align}\label{eq-basis-change}
		k^a & \mapsto k^a \, , & \delta^a_i & \mapsto \delta^a_i + \phi_i k^a \, , & \ell^a & \mapsto \ell^a - \phi^i \delta^a_i - \frac{1}{2} \phi^i \phi_i k^a \, .
	\end{align}
	
	In order to relate elements of the exterior algebra of $\V^*$ and those of the exterior algebra of $\mbb{H}_\K^*$, we state the following elementary lemma without proof.
	\begin{lem}\label{lem:useful}
		Any choice of vector $k$ in $\K$ establishes a one-to-one correspondence between elements of  $\bigwedge^p \mbb{H}_\K^*$ and elements $\phi$ of $\bigwedge^{p+1} \V^*$ such that $k \hook \phi = 0$ and $\kappa \wedge \phi = 0$, where $\kappa = g(k,\cdot)$. Explicitly,
		\begin{align*} 
			\varphi_{i_1 \ldots i_p} \mapsto \phi_{a_0 \ldots a_p}=(p+1) \kappa_{[a_0} \delta_{a_1}^{i_1} \ldots \delta_{a_p]}^{i_p} \varphi_{i_1 \ldots i_p} \, ,
		\end{align*} 
		for any inclusion $\delta^i_a : \mbb{H}^*_\K \rightarrow \V^*$.
		The inverse of this map is given by
		\begin{align*}
			\phi_{a_0 a_1 \ldots a_p} & \mapsto \varphi_{i_1 \ldots i_p} = \ell^{a_0} \delta^{a_1}_{i_1} \ldots \delta^{a_p}_{i_p} \phi_{a_0 a_1 \ldots a_p} \, ,
		\end{align*}
		for a $(p+1)$-covector  $\phi_{a_0 a_1 \ldots a_p}$  that satisfies $k^{a_0} \phi_{a_0 a_1 \ldots a_p} = \kappa_{[b} \phi_{a_0 a_1 \ldots a_p]}=0$ and
		where $\ell^a$ is any null vector such that $g_{a b} k^a \ell^b = 1$, and $\delta^i_a k^a = \delta^i_a \ell^a = 0$.
		
		In particular, we have an isomorphism of $\mathbf{Euc}^0(n)$-modules $\mbb{H}_\K^* \cong \g/\p$, where $\mathbf{Euc}^0(n)$ stabilises $k$.
	\end{lem}
	
	For each $w \in \mbf{R}$, we also introduce the one-dimensional representation $\R(w)$ of $P$ on $\R$ that is given by
	\begin{align}\label{eq-one-dim-rep}
		(a,A,\xi) \cdot r= a^w r \, ,
	\end{align}
	and restricts to a representation of $P_0$ on $\R$. In the future, we shall simply write $(\R^n)^* (w)$ for $(\R^n)^* \otimes \R(w)$ for any integer $w$.
	
	Now, any element $k^a$ of $\V^1$ can be seen as a map from $\V^*$ to $\R$ with kernel $(\V^*)^0$. In particular, we obtain the following isomorphisms
	\begin{align*}
		\gr_{-1}(\V^*) & \cong \R(1) \, , & \V^*_{-1} & \cong \R(1) \, ,
	\end{align*}
	of $P$-modules and $P_0$-modules respectively. Dually, we also obtain the isomorphism
	\begin{align*}
		\V^*_{1} & \cong \R(-1) \, ,
	\end{align*}
	of $P$-modules. More generally, for any non-negative integers $w, w'$, we have an isomorphism
	\begin{align*}
		{\bigotimes}^{w} (\V^*_{-1} )\otimes {\bigotimes}^{w'} (\V^*_1)& \cong \R(w-w') \, ,
	\end{align*}
	of $P_0$-modules.

	Finally, we note that the center $\mf{z}_0$ of $\g_0$ contains a distinguished \emph{grading element}
	\begin{align}\label{eq-grad-elem}
		E & = 
		\begin{pmatrix}
			-1 & 0 & 0 \\
			0 & 0 &  0 \\
			0 & 0 & 1
		\end{pmatrix}
	\end{align}
	with eigenvalues $\pm1$ and $0$ on $\V_{\pm1}$, $\g_{\pm1}$, and $\V_0$, $\g_0$ respectively. If $\mbb{A}$ is a submodule of the $\g_0$-module $\left( \otimes^p \V_1 \right) \otimes \left( \otimes^q \V_0 \right) \otimes \left( \otimes^r \V_{-1} \right)$, then $E$ has eigenvalue $p-r$ on $\mbb{A}$.

	\begin{rem}
		In dimension six, the isomorphism $\so(4) \cong \su^+(2) \oplus \su^-(2)$, where $\su^\pm(2)$ are two distinct copies of the Lie algebra of the special unitary group $\SU(2)$, leads to a further distinction to be made between $\g_0^+$ and $\g_0^-$, the self-dual part and anti-self-dual part of $\g_0$, each isomorphic to a copy of $\su(2)$.
	\end{rem}
	
	\subsection{The space $\mbb{G}$ of algebraic intrinsic torsions}
	We now consider the $P$-module
	\begin{align}\label{eq-intors-alg}
		\mbb{G} := \V^* \otimes (\g/\p) \, .
	\end{align}
	
	\begin{prop}\label{proposition-sim-intors}
		The $P$-module $\mbb{G} := \V^* \otimes \left( \g / \mf{p} \right)$ admits a filtration of $P$-modules
		\begin{align*}
			\mbb{G}^0 \subset \mbb{G}^{-1} \subset \mbb{G}^{-2}=\mbb{G}  \, ,
		\end{align*}
		where $\mbb{G}^i := (\V^*)^{i+1} \otimes \left( \g / \mf{p} \right)$ for $i=0,-1,-2$.
		The associated graded vector space
		\begin{align}\label{eq-ass-G-p}
			\gr(\mbb{G}) & = \gr_0 (\mbb{G}) \oplus \gr_{-1} (\mbb{G}) \oplus \gr_{-2} (\mbb{G}) \, , & & \mbox{where $\gr_i (\mbb{G}) :=  \mbb{G}^i / \mbb{G}^{i+1}$,}
		\end{align}
		decomposes as a direct sum
		\begin{align*}
			\gr_0 (\mbb{G}) & = \gr^0_0 (\mbb{G})=\mbb{G}^0, &
			\gr_{-1} (\mbb{G}) & = \gr^0_{-1} (\mbb{G}) \oplus \gr^1_{-1} (\mbb{G}) \oplus \gr^2_{-1} (\mbb{G}), & 
			\gr_{-2} (\mbb{G}) & = \gr_{-2}^0 (\mbb{G}), 
		\end{align*}
		of $P$-modules $\gr_i^j( \mbb{G})$, which are isomorphic to the $P_0$-modules $\mbb{G}_i^j$ in Table \ref{tab:table-irred-sim-mod-G}.
		\begin{center}
			\begin{table}
				\begin{displaymath}
					{\renewcommand{\arraystretch}{1.5}
						\begin{array}{||c|c|c|c|c||}
							\hline
							\text{$P_0$-module} & \text{Description} & \text{Dimension $n$} & \text{Dimension $n=2m$} & \text{Dimension $n=2m+1$} \\
							\hline
							\mbb{G}_{-2}^0 & (\R^n)^* (2) & n & 2m & 2m+1 \\
							\hline
							\mbb{G}_{-1}^0 &  (\R^n)^* (1)  & 1 & 1 & 1  \\
							\mbb{G}_{-1}^1 & \bigwedge^2 (\R^n)^* (1) & \frac{1}{2}n(n-1) & m(2m-1) & m(2m+1) \\
							\mbb{G}_{-1}^2 & \odot_\circ^2 (\R^n)^* (1) & \frac{1}{2}n(n+1) - 1 & (m+1)(2m-1) & m(2m+3)  \\
							\hline
							\mbb{G}_0^0 & (\R^n)^* & n & 2m & 2m+1  \\
							\hline
					\end{array}}
				\end{displaymath}
				\caption{\label{tab:table-irred-sim-mod-G}Irreducible $P_0$-submodules of $\mbb{G}$ for $n\neq 4$}
			\end{table}
		\end{center}
		These modules are all irreducibles except in the case $n=4$, where $\mbb{G}_{-1}^1$ splits further into two irreducibles $\mbb{G}_{-1}^{1,+}$ and $\mbb{G}_{-1}^{1,-}$ as shown in Table \ref{tab:table-irred-sim-mod-G-dim6}.
		\begin{center}
			\begin{table}
				\begin{displaymath}
					{\renewcommand{\arraystretch}{1.5}
						\begin{array}{||c|c|c||}
							\hline
							\text{$P_0$-module} & \text{Description} & \text{Dimension $n=4$} \\
							\hline
							\mbb{G}_{-2}^0 & (\R^n)^* (2)& 4 \\
							\hline
							\mbb{G}_{-1}^0 & (\R^n)^* (1)  & 1 \\
							\mbb{G}_{-1}^{1,+} & \bigwedge^{2,+} (\R^n)^* (1) & 3 \\
							\mbb{G}_{-1}^{1,-} & \bigwedge^{2,-} (\R^n)^* (1)   & 3 \\
							\mbb{G}_{-1}^2 &  \odot_\circ^2 (\R^n)^* (1)  &9 \\
							\hline
							\mbb{G}_0^0 & (\R^n)^*  & 4  \\
							\hline
					\end{array}}
				\end{displaymath}
				\caption{\label{tab:table-irred-sim-mod-G-dim6}Irreducible $P_0$-submodules of $\mbb{G}$ for $n=4$}
			\end{table}
		\end{center}

	\end{prop}
	
	\begin{proof}
		We first note that the filtration \eqref{eq-sim-filtration-dual} of $P$-modules on $\V^*$ induces a $P$-invariant filtration
		on $\mbb{G}$. Each summand of the associated graded vector space \eqref{eq-ass-G-p} is isomorphic to a completely reducible $P_0$-module. To describe them, we work in the splitting \eqref{eq-sim-grading-dual}, and express $\mbb{G}$ in terms of its corresponding $P_0$-modules, i.e.
		\begin{align*}
			\mbb{G} & \cong ( \V^*_{-1} \oplus \V^*_0 \oplus \V^*_{1} ) \otimes \g_{-1} \, .
		\end{align*}
		Since $\g_{-1} \cong \V^*_{-1} \otimes \V^*_0$, we immediately have the vector spaces isomorphisms
		\begin{align*}
			\gr_{-2} ( \mbb{G}) & \cong \V^*_{-1} \otimes \V^*_{-1} \otimes \V^*_0  \, , & 
			\gr_{-1} ( \mbb{G}) & \cong \otimes^2 \V^*_0 \otimes \V^*_{-1} \, , &
			\gr_0 ( \mbb{G}) & \cong \V^*_0 \, .
		\end{align*}
		Clearly, $\mbb{G}_0^0 := \gr_0 ( \mbb{G})$ and  $\mbb{G}_{-2}^0 := \gr_{-2} ( \mbb{G})$ are irreducibles. On the other hand, it is easy to check that
		\begin{align*}
			\otimes^2 \V^*_0 \otimes \V^*_{-1} = \V^*_{-1} \oplus \left( \bigwedge^2 \V^*_0 \otimes \V^*_{-1} \right) \oplus \left( \odot_\circ^2 \V^*_0 \otimes \V^*_{-1} \right) \, .
		\end{align*}
		This yields the decomposition of 
		$\gr_{-1} ( \mbb{G})$ as the direct sum of the irreducible $P$-modules $\mbb{G}_{-1}^0$, $\mbb{G}_{-1}^1$ and $\mbb{G}_{-1}^2$.
		This completes the proof for $n\neq4$. The case $n=4$ follows from the fact that $\g_0 \cong \bigwedge^2 \V^*_0$ splits into a self-dual part and anti-self-dual part.
	\end{proof}

	Before we proceed, we fix a splitting of $\mbb{G}$ into $P_0$-modules, and introduce, for each $i,j$, a $P_0$-module epimorphism $\Pi_i^j : \V^* \otimes \g \rightarrow \mbb{G}_i^j$ with the properties that $\V^* \otimes \mf{p}$ lies in the kernel of $\Pi_i^j$, and $\Pi_i^j$ descends to a projection from $\mbb{G}$ to $\mbb{G}_i^j$. For convenience, we shall use the co-vector $\kappa_a = g_{a b} k^b$. We streamline notation by setting $(\Gamma \cdot \kappa)_{ab} := -\Gamma_{ab}{}^{c} \kappa_c$ for any element $\Gamma_{abc}$ of $\V^* \otimes \g$, and $\kappa_a \in \Ann(\K^\perp)$. Note that $(\Gamma \cdot \kappa)_{ab} k^b = 0$ , i.e., $(\Gamma \cdot \kappa)_{ab} \in \V^* \otimes \Ann(\K)$. More explicitly, we define the projections
	\begin{align*}
		\Pi_{-1}^0 & : \V^* \otimes \g \rightarrow \mbb{G}_{-1} \, , && \Gamma_{ab}{}^{c} \mapsto \Pi_{-1}(\Gamma)_{i j} :=  (\Gamma \cdot \kappa)_{ab} \delta^a_{i} \delta^b_{j}  \, ,
	\end{align*}
	where $\mbb{G}_{-1} = \mbb{G}_{-1}^0 \oplus \mbb{G}_{-1}^1 \oplus \mbb{G}_{-1}^2$,
	and
	\begin{align}\label{eq-proj-intors-sim}
		\begin{aligned}
			\Pi_{-2}^0 & : \V^* \otimes \g \rightarrow \mbb{G}_{-2}^0 \, , && \Gamma_{ab}{}^{c} \mapsto \Pi_{-2}^0(\Gamma)_i := (\Gamma \cdot \kappa)_{ab} k^a \delta^b_i  \, , \\
			\Pi_{-1}^0 & : \V^* \otimes \g \rightarrow \mbb{G}_{-1}^0 \, , && \Gamma_{ab}{}^{c} \mapsto \Pi_{-1}^0(\Gamma) :=  \Pi_{-1}(\Gamma)_{i j} h^{i j} \, , \\
			\Pi_{-1}^1 & : \V^* \otimes \g \rightarrow \mbb{G}_{-1}^1 \, , && \Gamma_{ab}{}^{c} \mapsto \Pi_{-1}^1 (\Gamma)_{ij} := \Pi_{-1}(\Gamma)_{[i j]}  \, , \\
			\Pi_{-1}^2 & : \V^* \otimes \g \rightarrow \mbb{G}_{-1}^2 \, , && \Gamma_{ab}{}^{c} \mapsto \Pi_{-1}^2 (\Gamma)_{ij} := \Pi_{-1}(\Gamma)_{(i j)_\circ}  \, , \\
			\Pi_0^0 & : \V^* \otimes \g \rightarrow \mbb{G}_0^0 \, , && \Gamma_{ab}{}^{c} \mapsto 
			\Pi_0^0 (\Gamma)_i := (\Gamma \cdot \kappa)_{ab} \ell^a \delta^b_i  \, .
		\end{aligned}
	\end{align}
	In dimension six, we also define
	\begin{align}\label{eq-proj-intors-sim6}
		\Pi_{-1}^{1,\pm} & : \V^* \otimes \g \rightarrow \mbb{G}_{-1}^{1,\pm} \, , && \Gamma_{ab}{}^{c} \mapsto \Pi_{-1}^{1,\pm} (\Gamma)_{ij} := \frac{1}{2} \left( \Pi_{-1}^{1}(\Gamma)_{i j}  \pm \star \Pi_{-1}^{1}(\Gamma)_{i j} \right) \, ,
	\end{align}
	where $\star$ is the Hodge star on $(\V^*)_0$, i.e., $\Pi_{-1}^{1}(\Gamma)_{i j} =  \Pi_{-1}^{1}(\Gamma)_{k \ell} \frac{1}{2} (\varepsilon_K)_{i j}{}^{k \ell}$, 
	
	By construction, for each $i, j$, the kernel of $\Pi_i^j \mod \V^* \otimes \mf{p}$ is precisely isomorphic to the complement $(\mbb{G}_i^j)^c$ of $\mbb{G}_i^j$ in $\mbb{G}$ as a $P_0$-module, i.e.\
	\begin{align}\label{eq-proj}
		\ker \Pi_i^j / (\V^* \otimes \mf{p}) & \cong  (\mbb{G}_i^j )^c \, , & \mbox{i.e.} & & \left( \ker \Pi_i^j /( \V^* \otimes \mf{p} )\right)^c & \cong \mbb{G}_i^j \, .
	\end{align}
	
	Now, once a splitting of $\mbb{G}$ is chosen, any $P$-submodule of $\mbb{G}$ must be a sum of the irreducible $P_0$-submodules $\mbb{G}_i^j$ given in Proposition \ref{proposition-sim-intors}.  Clearly, not every such sum is a $P$-module. This is clear since the abelian part $P_+$ of $P$ sends elements in $\mbb{G}_{i-1}$ to $\mbb{G}_i \supset \mbb{G}_i^j$. To determine which $P_0$-submodules of $\mbb{G}$ are also $P$-submodules, we use the characterisation \eqref{eq-proj}, and compute how a change of splitting \eqref{eq-basis-change} affects the maps $\Pi_i^j$. Such transformations will tell us how the various modules $\mbb{G}_i^j$ are related under the action of $P_+$, and we will be able to determine the $P$-submodule of $\mbb{G}$ accordingly. This will be made clear in the next proposition.
	
	\begin{prop}\label{prop-intors-sim}
		The following 
		\begin{align*}
			\slashed{\mbb{G}}_{-2}^0 & := \{ \Gamma \in \V^* \otimes \g : \Pi_{-2}^0 (\Gamma) = 0 \}  / \V^*\otimes \mf{p} \, , \\
			\slashed{\mbb{G}}_{-1}^i & := \{ \Gamma \in \V^* \otimes \g : \Pi_{-2}^0 (\Gamma) = \Pi_{-1}^i (\Gamma) = 0 \} / \V^*\otimes \mf{p} \, , & & i = 0, 1, 2 \, ,
		\end{align*}
		and, in dimension six only, i.e.\ $n=4$,
		\begin{align*}
			\slashed{\mbb{G}}_{-1}^{1,\pm} & := \{ \Gamma \in \V^* \otimes \g : \Pi_{-2}^0 (\Gamma) = \Pi_{-1}^{1,\pm} (\Gamma) = 0 \} / \V^*\otimes \mf{p} \, ,
		\end{align*}
		are $P$-submodules of $\mbb{G}$, where each $\Pi_i^j$ is defined by \eqref{eq-proj-intors-sim} and $\Pi_{-1}^{1,\pm} $ by \eqref{eq-proj-intors-sim6}. In addition,
		\begin{align*}
			\slashed{\mbb{G}}_0^0 & := \{ \Gamma \in \V^* \otimes \g : \Pi_{-2}^2 (\Gamma) = \Pi_{-1}^i (\Gamma) = \Pi_0^0 (\Gamma) = 0 \, , i = 0, 1, 2 \} / \V^*\otimes \mf{p} \cong \{ 0 \} \, .
		\end{align*}
		Any $P$-submodule of $\mbb{G}$ arises as an intersection of any of $\mbb{G}$, $\slashed{\mbb{G}}_{-2}^0$, $\slashed{\mbb{G}}_{-1}^i$ for $i=1,2,3$, and in dimension six only, of $\slashed{\mbb{G}}_{1}^{1,\pm}$.
	\end{prop}
	
	\begin{proof}
		Let $\Gamma \in \V^* \otimes \g$, and set
		\begin{align}\label{eq-im-proj}
			\gamma_i & := \Pi_{-2}^0(\Gamma)_i \, , &
			\epsilon & := \Pi_{-1}^0(\Gamma) \, , &
			\tau_{ij} & := \Pi_{-1}^1(\Gamma)_{ij} \, , &
			\sigma_{ij} & := \Pi_{-1}^2(\Gamma)_{ij} \, ,  &
			E_i & := \Pi_0^0(\Gamma) \, .
		\end{align}
		Under the transformation \eqref{eq-basis-change}, the elements \eqref{eq-im-proj} transform as
		\begin{align*}
			\gamma_i & \mapsto \gamma_i \, , \\
			\epsilon & \mapsto \epsilon +  \gamma_i \phi^i \, , \\
			\tau_{ij} & \mapsto \tau_{ij} - \gamma_{[i} \phi_{j]} \, , \\
			\sigma_{ij} & \mapsto \sigma_{ij} + \gamma_{(i} \phi_{j)_\circ}  \, , \\
			E_i & \mapsto E_i -  \sigma_{i j}  \phi^j + \tau_{i j}  \phi^j - \frac{1}{n} \epsilon \phi_i  - \frac{1}{2} \phi^k \phi_k \gamma_i \, .
		\end{align*}
		By inspection, we immediately conclude that $\gr_{-2}^0 (\mbb{G})^c$ and $\gr_{-1}^i (\mbb{G})^c$, $i=0,1,2$ are $P$-modules.
		
		Now, suppose that $\mbb{A}$ is a proper $P$-submodule of $\mbb{G}$, i.e., $\mbb{A}$ is neither trivial or $\mbb{G}$. Then in a splitting, $\mbb{A}$ can be described as the intersection of the kernels of some of the projections $\Pi_i^j$. In particular, $\mbb{A}$ is a vector subspace of $\ker \Pi_i^j$ for some $i, j$. But since $\mbb{A}$ is a $P$-submodule, it must also be contained in the kernel of any other projections describing $\gr_{-2}^0 (\mbb{G})^c$ or $\gr_{-1}^i (\mbb{G})^c$, $i=0,1,2$. Hence, $\mbb{A}$ must be a $P$-submodule of any of $\gr_{-2}^0 (\mbb{G})^c$ and $\gr_{-1}^i (\mbb{G})^c$, $i=0,1,2$.
		
		The six-dimensional case, i.e.\ $n=4$, can be refined further by taking into account the splitting of $\bigwedge^2(\V^*)_0$ into its self-dual part and anti-self-dual part.
	\end{proof}

	%
	%

	\section{Optical geometry}\label{sec:geometry}
	Let $(\mc{M},g)$ be an oriented and time-oriented Lorentzian manifold of dimension $n+2$, and the metric $g$ will be assumed to have signature $(n+1,1)$, i.e., $g$ has only one negative eigenvalue. The volume form associated to $g$ will be denoted by $\varepsilon$, and the Levi-Civita connection $\nabla$. The Lie derivative along a vector field $v$ will be denoted $\mathsterling_v$.
	In addition, the divergence of a vector field $w$ will be denoted $\div \, w$.
	If $E \subset T \mc{M}$ is a vector distribution, we shall denote its annihilator by $\Ann(E) \subset T^* \mc{M}$. For two $1$-forms $\kappa$ and $\lambda$ we denote by $\kappa\,\lambda$ their symmetric product $\kappa \odot \lambda = \frac{1}{2} \left( \kappa \otimes \lambda +\lambda \otimes \kappa \right)$, or in abstract index notation, $\kappa_{(a} \lambda_{b)} = \frac{1}{2} \left( \kappa_{a} \lambda_{b} + \kappa_{b} \lambda_{a} \right)$.
	
	\subsection{Optical structures}
	From remarks in Section \ref{algsec}  it follows that the possible indecomposable
	connected holonomy groups are either equal to $\SO^0(n+1,1)$ or contained in
	the maximal parabolic $\Sim^0(n)$ in $\SO^0(n+1,1)$. This motivates the following definition:
	\begin{defn}\label{optdef}
		Let $(\mc{M},g)$ be an oriented and time-oriented Lorentzian manifold of dimension $n+2$. An \emph{optical structure} on $(\mc{M},g)$ is given by a vector distribution $K\subset T\mc{M}$ of tangent null lines. We refer to $(\mc{M},g,K)$ as an \emph{optical geometry}.
	\end{defn}
	We note that the assumption that $(\mc{M},g)$ is time-oriented implies that $K$ is also oriented (see for example \cite{Baum2014}).

	Denoting by $K^\perp$ the orthogonal space of $K$ with respect to $g$, we have a filtration
	\begin{align}\label{eq-K-filt}
		K \subset K^\perp \subset T \mc{M} \, ,
	\end{align}
	and short exact sequences
	\begin{align}
		& 0 \longrightarrow K \longrightarrow  K^\perp \longrightarrow  H_K \longrightarrow  0 \, , \label{eq-seseq-K} \\
		& 0 \longrightarrow \Ann(K^\perp) \longrightarrow  \Ann(K) \longrightarrow  H^*_K \longrightarrow  0 \, , \label{eq-seseq-AnnK} \\
		& 0 \longrightarrow H_K \longrightarrow  T \mc{M} / K \longrightarrow  T \mc{M} / K^\perp \longrightarrow  0 \, , \label{eq-seseq-TM}
	\end{align}
	of vector bundles on $\mc{M}$. Equation \eqref{eq-seseq-K} defines the \emph{screen bundle} as the rank-$n$ vector bundle $H_K = K^\perp/ K$, and equation \eqref{eq-seseq-AnnK} with its dual $H_K^* = \Ann(K)/\Ann(K^\perp)$. Since $K$ is null, $\Ann(K) \cong K^\perp$ and $\Ann(K^\perp) \cong K$. The isomorphism $H_K \cong H_K^*$ is established by means of the positive definite bundle metric $h$ induced from $g$. It is given by
	\begin{align}\label{eq-metric-H}
		h (v + K , w + K ) & := g(v,w)  \, , & \mbox{for any $v, w \in \Gamma(K^\perp)$.}
	\end{align}
	Note, that $H_K$ also inherits a volume element $\varepsilon_K$ associated to $h$ given by
	\begin{align}\label{eq-K-vol}
		\varepsilon_K ( v_1 + K \, , \ldots , v_n + K ) \kappa & = \varepsilon ( k, v_1 \, , \ldots \, , v_n \, , \cdot ) \, , & \mbox{for any $v_1 \, , \ldots \,  , v_n \in \Gamma(K^\perp)$,}
	\end{align}
	for any $k \in \Gamma(K)$ and where $\kappa = g(k, \cdot)$. This is clearly independent of the choice of section of $K$.
	
	\begin{rem}\label{term-remark}
		The notion of an `optical structure' has been used in the literature for several different structures related to, but  not necessarily coinciding with, our Definition~\ref{optdef}.
		For example, in \cite[Section 1.4]{nurowski93-phd} and \cite{nurowski96} Nurowski defined an (almost) optical structure on a Lorentzian manifold of even dimension as a vector distribution of null lines together with a bundle complex structure $J$ on the screen bundle $H_K$ that is compatible with the metric $h$. An optical structure was then defined as an  almost optical structure satisfying certain integrability conditions. For these (almost) optical structures later in  \cite{Nurowski2002} Nurowski and Trautman introduced the name {\em (almost) Robinson structures}.  We will  use the notion of an (almost) Robinson structure in the same sense in the forthcoming \cite{Fino2023} (see Remark \ref{rob-remark}) and  use the weaker notion of optical structure as in Definition \ref{optdef} to denote the existence of a null line distribution $K$.
		
		On four-dimensional Lorentzian manifolds the orientability of the screen bundle, which in this case is of rank two, is sufficient for the existence of a compatible complex structure on the rank-two screen bundle, so here our notion of optical structure coincides with Nurowski's and with the notion of  an almost Robinson structure.
		
		Another notion of an optical structure was introduced by Robinson and Trautman in \cite{Robinson1985,Trautman1985, Robinson1986,Robinson1989} as a flag geometry on a smooth manifold given by a filtration of the tangent bundle by a line subbundle and a co-rank 1 subbundle together with a certain equivalence class of Lorentzian metrics for which the line bundle is null (see \cite[Definition 1]{Robinson1989}). In Definition \ref{def:gen_opt_str} we will call such structures  {\em generalised optical structures}.
		
		Finally, the third author of the present article also used the terminology `optical structure' on a \emph{five}-dimensional Lorentzian manifold in reference \cite{Taghavi-Chabert2011} to describe a null line distribution that arises from a totally null complex distribution of rank two. This may be described more accurately as a five-dimensional analogue of an almost Robinson structure.
	\end{rem}

	For a given optical structure $(\mc{M},g,K)$ a  choice of null line distribution $L$ dual to $K$ splits the filtration \eqref{eq-K-filt} as a direct sum of vector bundles
	\begin{align}\label{eq-K-split}
		T \mc{M} & = L \oplus \left( K^\perp \cap L^\perp \right) \oplus K \, .
	\end{align}

	We shall use the convention already introduced in Section \ref{sec:algebra}: upper lower-case Roman indices from the beginning of the alphabet $a, b, \ldots$ will refer to sections of $T\mc{M}$, while those starting from the middle of the alphabet $i,j,k,\ldots$ will refer to sections of either $K^\perp/K$ or $K^\perp \cap L^\perp$ once a complement $L$ to $K^\perp$ is chosen. Lower versions of these index types will refer to dual sections.
	
	Once such a splitting is chosen, we fix sections $k^a$ of $K$ and $\ell^a$ of $L$ such that $g_{ab} k^a \ell^b = 1$, together with injective maps $\delta^a_i : \Gamma( K^\perp \cap L^\perp )\rightarrow \Gamma( T \mc{M} )$. Viewing $\delta^a_i$ dually as projections from $T^* \mc{M}$ to $(K^\perp \cap L^\perp)^*$, the screen bundle metric $h_{i j}$ is given by $h_{i j} = g_{a b} \delta^a_i \delta^b_j$. Setting $\kappa_a = g_{a b} k^b$, $\lambda_a = g_{a b} \ell^b$ and $\delta^i_a = \delta_j^b g_{a b} h^{j i}$, where $h^{i j}$ is the inverse screen bundle metric, the metric $g$ takes the form
	\begin{align*}
		g_{a b} & = 2 \kappa_{(a} \lambda_{b)} + h_{i j} \delta^i_a \delta^j_b \, ,
	\end{align*}
	We note that the following transformations
	\begin{align}
		k^a & \mapsto \wt{k}^a = \e^{\varphi} k^a \, , & \delta^a_i & \mapsto \wt{\delta}^a_i = \delta^a_i \, , & \ell^a & \mapsto \wt{\ell}^a = \e^{-\varphi} \ell^a \, ,  \label{eq-boost} \\
		k^a & \mapsto \wt{k}^a = k^a \, , & \delta^a_i & \mapsto \wt{\delta}^a_i = \delta^a_i + \phi_i k^a \, , & \ell^a & \mapsto \wt{\ell}^a = \ell^a - \phi^i \delta^a_i - \frac{1}{2} \phi^i \phi^j h_{i j} k^a \, , \label{eq-null_rotation}
	\end{align}
	where $\varphi \in C^\infty(\mc{M})$ and $\phi_i \in C^\infty(\mc{M}\times \R^n)$, leave the form of the metric unchanged. Transformation \eqref{eq-boost} is referred to as a \emph{boost}, and \eqref{eq-null_rotation} as a \emph{null rotation}.
	
	\begin{rem}
		In the treatment of null frames in general  relativity it is customary to introduce a third type of transformations referred to as \emph{spin rotations} which are simply the $\SO(n)$ transformations acting on the orthogonal basis of $K^\perp \cap L^\perp$. In this article, we shall do away with such transformations since we will be using $\delta^a_i$ as abstract injections from $K^\perp \cap L^\perp$ to $T \mc{M}$.
	\end{rem}

	\subsubsection{Optical structures as $G$-structures}
	In accordance with the definition of an $H$-structure given in Section \ref{sec:G-structures}, an optical structure defines a $\Sim^0(n)$-structure, i.e.,  a reduction of the frame bundle to the $\Sim^0(n)$-bundle $\mc F^P$, with $P=\Sim^0(n)$, whose fiber over any point $p \in \mc{M}$ is given by
	\begin{align*}
		\mc F^{P}_p=\{(e_0,\ldots , e_{n+1})|_p \in \mc F_p \mid e_{n+1}|_p\in K_p\} \, .
	\end{align*}
	Note that if we drop the assumption that $\mc{M}$ is time-oriented, $K$ is then not necessarily oriented, and one merely has a $\Sim(n)$-structure.
	
	If $\mbb{A}$ is a $P$-module, the corresponding associated vector bundle is given by
	\begin{align*}
		\mc{F}^P (\mbb{A}) & := \mc{F}^P \times_P \mbb{A} \, .
	\end{align*}
	More specifically, we have the following line bundles:
	\begin{defn}
		For each $w \in \R$, we defined the \emph{bundle of boost densities of weight $w$} to be the line bundle
		\begin{align*} \mc E(w) =\mc F^{P}( \R(w)) ,\end{align*} 
		where $\R(w)$ is the one-dimensional representation of $P$ on $\R$ given by \eqref{eq-one-dim-rep}.  We shall say that a section $s$ of $\mc{E}(w)$ has \emph{boost weight} $w$.
	\end{defn}

	A section $s$ of $\mc{E}(w)$ satisfies the following property: two distinct sections $k^a$ and $\wt{k}^a$ of $K$ with $\wt{k}^a = a k^a$ for some non-vanishing function $a$ induce two distinct trivialisations of $\mc{E}(w)$ with respect to which $s = f$ and $s = \wt{f}$, with $\wt{f} = a^{-w} f$.
	
	Being an associated vector bundle to the frame bundle, the Levi-Civita connection extends to a connection on $\mc{E}(w)$, which we shall also denote $\nabla$.
	
	We note that for any vector bundle $F$, it will be convenient to adopt the short-form notation $F(w)$ for $F \otimes \mc{E}(w)$.
	
	It is also straightforward to make the following identification of $P$-invariant vector bundles
	\begin{align}\label{eq-K-density}
		K & \cong \mc{E}(-1) \, .
	\end{align}
	
	\subsection{Intrinsic torsion}
	With reference to equation \eqref{eq-intors-gen} and Section \ref{sec:algebra}, let us set $\mc{G} = \mc{F}^P(\mbb{G})$ where $\mbb{G}$ given by \eqref{eq-intors-alg}. Then the intrinsic torsion $\mathring{T}$
	of a $\Sim^0(n)$-structure is given  by  a section of $\mc{G} \cong T^* \mc{M}\otimes \mc F^{P}(\g/\p)$ at each point $p \in M$ defined by
	\begin{align*} 
		\mathring{T}(v)|_p &= 
		\left[ (\ell,e_1,\ldots, e_n,k) \, ,
		\pi
		\begin{pmatrix} * & -g(\nabla^g_v k, e_i) & 0 \\
			* & * & g(\nabla^g_v k,e_i)\\
			0 & * & *
		\end{pmatrix}
		\right] |_p\, ,
	\end{align*} 
	for any $v \in \Gamma(T \mc{M})$ and any  local section $(k,e_1,\ldots, e_n, \ell )$ of $\mc F^{P}$.
	Note that we have $\nabla_v k \in \Gamma(K^\perp)$. We have the following isomorphism of vector bundles
	\begin{eqnarray*}
		\mc F^{P}(\g/\p) &\simeq & H_K(1),
		\\
		\left[ (\ell, e_1, \ldots , e_n , k), \pi\begin{pmatrix} *&(-x^i)&0\\ * & * & (x^i) \\0&*&*
		\end{pmatrix}\right]
		& \mapsto & 1 \otimes [x^i e_i].
	\end{eqnarray*}
	This implies that with a choice of trivialisation $k \in \Gamma(K)$, the intrinsic torsion of a $\Sim^0(n)$-structure can be identified with 
	\begin{align*}
		\mathring{T}_{a j} = (\nabla_a \kappa_b) \delta^b_j \in \Gamma(T^* \mc{M} \otimes H_K^* )\, , &&  \mbox{where} & & \kappa_{a} = g_{a b} k^{b} \, .
	\end{align*} 

	We shall be concerned with characterising the intrinsic torsion according to the algebraic analysis of Section \ref{sec:algebra}. For this purpose, let us define the associated vector bundles
	\begin{align*}
		\mc{G}^i := \mc{F}^P \times_P \mbb{G}^i \, , 
	\end{align*}
	so that there is a filtration of bundles
	\begin{align*}
		\mc{G} =: \mc{G}^{-2} \supset\mc{G}^{-1} \supset \mc{G}^0 \, ,
	\end{align*}
	with associated graded vector bundle
	\begin{align*}
		\gr(\mc{G}) = \gr_{-2}^0 (\mc{G}) \oplus \left( \gr_{-1}^0 (\mc{G}) \oplus \gr_{-1}^1 (\mc{G}) \oplus \gr_{-1}^2 (\mc{G}) \right) \oplus \gr_0^0 (\mc{G}) \, .
	\end{align*}
	Once a splitting is chosen, we also obtain the $P_0$-invariant vector bundles $\mc{G}_i^j := \mc{F}^P \times_{P_0} \mbb{G}_i^j$, where we recall $P_0 = \CO^0(n)$. These can be identified as follows:
	\begin{align*}
		\mc{G}_{-2}^0 & \cong H_K^*(2) \, , \\
		\mc{G}_{-1}^0 & \cong \mc{E}(1) \, , &
		\mc{G}_{-1}^1 & \cong \bigwedge^2 H_K^*(1) \, , &
		\mc{G}_{-1}^2 & \cong \odot^2_\circ H_K^*(1) \, ,  \\
		\mc{G}_{0}^0 & \cong H_K^*\, .
	\end{align*}
	We extend the projections $\Pi_i^j$ defined by \eqref{eq-proj-intors-sim} to bundle projections in the obvious way: let $k \in \Gamma(K)$, set $\kappa = g(k,\cdot)$, choose $\ell$ such that $g(k,\ell)=1$ and denote by $\delta^a_i$ the projections onto the complements of the spans of $k$ and $\ell$. Then we may choose a connection $1$-form $\Gamma_{ab} \,{}^c$ such that $\nabla_a \kappa_b = - \Gamma_{ab}\,{}^c \kappa_c$, we then have
	\begin{align*}
		(\nabla_a \kappa_b) k^a \delta^b_i & = \Pi_{-2}^0(\Gamma)_i \, , \\ 
		(\nabla_a \kappa_b) \delta^a_i \delta^b_j h^{i j} & = \Pi_{-1}^0(\Gamma) \, , & 
		(\nabla_a \kappa_b) \delta^a_{[i} \delta^b_{j]} & = \Pi_{-1}^1(\Gamma)_{ij} \, ,& 
		(\nabla_a \kappa_b) \delta^a_{(i} \delta^b_{j)} & = \Pi_{-1}^2(\Gamma)_{ij} \, ,\\ 
		(\nabla_a \kappa_b) \ell^a \delta^b_i & = \Pi_0^0(\Gamma).
	\end{align*}
	Armed with these, we can characterise sections of the $P$-invariant vector bundles $\slashed{\mc{G}}_i^j := \mc{F}^P \times_P \slashed{\mbb{G}}_i^j$ using the algebraic decomposition of $\mbb{G}$ given in Proposition \ref{prop-intors-sim}. For instance, the intrinsic torsion $\mathring{T}$ is a section of $\slashed{\mc{G}}_{-1}^1$ if and only if $(\nabla_a \kappa_b) k^a \delta^b_i = (\nabla_a \kappa_b) \delta^a_{[i} \delta^b_{j]} = 0$. In the next sections, we shall give alternative characterisations of the intrinsic torsion.

	\begin{rem}
		In principle, the intrinsic torsion of a given optical structure may change classes from point to point. For simplicity, we shall assume that it does not.
	\end{rem}
	
	\subsection{Congruence of null curves}
	Under the assumption that $\mc{M}$ is time-oriented (so that $K$ is oriented), the line distribution gives rise to a \emph{congruence of (oriented) null curves}, that is, a foliation of $\mc{M}$ by unparametrised oriented curves tangent to $K$. Conversely, any optical geometry arises in this way. We shall typically denote such a congruence $\mc{K}$. Any vector field tangent to the curves of $\mc{K}$ will be referred to as a \emph{generator} of $\mc{K}$.

	Each local section $k$ of $K$ induces a parametrisation $t$ along each curve of the congruence, i.e., $k=\frac{\partial}{\partial t}$. Any other section of $K$ must differ from $k$ by some positive smooth factor of $K$, and gives rise to the same congruence of oriented null curves, but with a different parametrisation.

	Given a congruence of null curves $\mc{K}$, it will be particularly convenient to consider the local leaf space $\ul{\mc{M}}$ of $\mc{K}$ in an open subset $\mc{U}$ of $\mc{M}$, i.e., a surjective submersion $\varpi$ from $\mc{U}$ to an $(n+1)$-dimensional smooth manifold $\ul{\mc{M}}$: for each $x \in \ul{\mc{M}}$, the inverse image $\varpi^{-1} (x)$ is a null curve in $\mc{U}$ tangent to $K$. This leaf space will feature prominently in the study of optical structures as we shall see later. Again the considerations here will be essentially local.

	If any geometric structure on $\mc{M}$ is preserved along the flow of some generator of $\mc{K}$, then it must descend to the leaf space. In general, the leaf space of a congruence of null curves will have no structure, not even an orientation.
	
	\subsubsection{Congruence of null geodesics}
	We now turn to the case where the intrinsic torsion of an optical geometry $(\mc{M},g, K)$ satisfies the weakest possible non-trivial condition.
	
	Recall that the null curves of a congruence $\mc{K}$ tangent to $K$ are geodesics if
	\begin{align}\label{eq-null-geod}
		\nabla_k k & = f k \, , &  \mbox{for some $k \in \Gamma(K)$ and $f \in C^\infty(\mc{M})$.}
	\end{align}
	This condition clearly does not depend on the choice of generator of $\mc{K}$. Note that one can always find local generators of $\mc{K}$ that satisfy
	\begin{align}\label{eq-aff-geod}
		\nabla_k k = 0 \, .
	\end{align}
	The integral curves of such a $k$ are \emph{affinely parametrised} geodesics:
	if $t \in \R$ is an affine parameter along a null geodesic, any affine reparametrisation $t \mapsto a t + b$ for some smooth functions $a$ and $b$ constant along $k$, where $a$ is non-vanishing, induces the transformation \eqref{eq-boost},  where $\wt{k}$ also generates  affinely parametrised geodesics.

	A more convenient characterisation of the congruence of null geodesics is given in the following lemma:
	\begin{lem}[\cite{Robinson1983}] \label{lem-proposition-geod}
		Let $(\mc{M},g,K)$ be an optical structure with congruence of null curves $\mc{K}$. The following statements are equivalent:
		\begin{enumerate}
			\item the curves of $\mc{K}$ are geodesics;\label{item:K-geod}
			\item the distribution $K^\perp$ is preserved along the flow of any generator $k$ of $\mc{K}$, i.e., $\mathsterling_k v \in \Gamma(K^\perp)$ for any $v \in \Gamma(K^\perp)$, or equivalently, $\mathsterling_k \kappa \in \Gamma(\Ann(K^\perp))$,  where $\kappa = g(k,\cdot)$.\label{item:kap-geod}
		\end{enumerate}
		Further, if any of \eqref{item:K-geod} and \eqref{item:kap-geod} holds, and the geodesics of $\mc{K}$ are affinely parametrised by $k$, then
		\begin{align*}
			\mathsterling_k \kappa = 0 \, .
		\end{align*}
		In particular, the $1$-form $\kappa$ descends to a $1$-form on the leaf space $\ul{\mc{M}}$ of $\mc{K}$ where it annihilates a rank-$n$ distribution $\ul{H}$, i.e., locally $\kappa$ is the pull back of a $1$-form $\ul{\kappa}$ annihilating $\ul{H}$.
	\end{lem}
	
	\begin{proof}
		This is a straightforward computation: for any vector field $v \in \Gamma(K^\perp)$, we have $\mathsterling_k \kappa (v) = \nabla_k \kappa (v) + \kappa(\nabla_v k) = g(\nabla_k k,v) = f \kappa (v)$ by \eqref{eq-null-geod}.
	\end{proof}
	\vspace{2.5mm}

	The relation to the intrinsic torsion now follows immediately:
	\begin{prop}
		Let $(\mc{M}, g, K)$ be an optical geometry with congruence of null curves $\mc{K}$. Let $\mathring{T} \in \Gamma ( \mc{G} )$ be the intrinsic torsion of the optical structure. Then the curves of $\mc{K}$ are geodesics if and only if $\mathring{T} \in \Gamma( \mc{G}^{-1}  ) = \Gamma( \slashed{\mc{G}}_{-2}^0  )$.
	\end{prop}
	
	In other words, the obstruction to the curves of $\mc{K}$ being geodesics is given by a weighted section of $H_K^*(2)$, which we shall denote $\breve{\gamma}_i$.
	
	Thus, for a congruence of null geodesics, the short exact sequences \eqref{eq-seseq-AnnK} and \eqref{eq-seseq-TM} on $\mc{M}$ give rise to short exact sequences on $\ul{\mc{M}}$:
	\begin{align}
		& 0 \longrightarrow \Ann(\ul{H}) \longrightarrow  T^* \ul{\mc{M}} \longrightarrow  \Ann(T \ul{\mc{M}} / \ul{H}) \longrightarrow  0 \, , \label{eq-seseq-AnnKudl} \\
		& 0 \longrightarrow \ul{H}  \longrightarrow  T \ul{\mc{M}} \longrightarrow  T \ul{\mc{M}} / \ul{H} \longrightarrow  0 \, . \label{eq-seseq-Kudl}
	\end{align}
	In general, one can identify sections of $\Ann(\ul{H})$ with those of $\Ann(H_K)$: let $k$ be a generator of $\mc{K}$ that defines an affine parameter for the geodesics, then $\kappa = g(k, \cdot)$ descends to a section of $\Ann(\ul{H})$. Conversely, if $\kappa$ is the pull-back of some section $\ul{\kappa}$ of $\Ann(\ul{H})$, then it is covariantly constant along $k^a = g^{a b} \kappa_b$, and thus $k^a$ defines an affine parameter for the geodesics.

	\subsubsection{Optical invariants}
	It is well-known, see e.g.\ \cite{Robinson1983,Penrose1986,Stephani2003,Pravda2004} and references therein, that a congruence of null geodesics has a number of invariants. We presently relate them to the intrinsic torsion of the optical structure.
	
	\begin{defn}\label{definition-tw-sh-exp}
		Let $(\mc{M},g,K)$ be an optical geometry with associated congruence of null geodesics $\mc{K}$. Let $k$ be a generator of $\mc{K}$ and $\kappa = g(k,\cdot)$. The sections $\epsilon \in C^\infty (\mc{M})$, $\tau \in \Gamma(\bigwedge^2 H^*_K)$ and $\sigma \in \Gamma( \odot^2_\circ H^*_K )$ given respectively by
		\begin{subequations}\label{eq-ETS}
			\begin{align}
				\epsilon \, \kappa & =   \kappa \, \div  k -  \nabla_{k} \kappa   \, , \label{eq-exp} \\
				\tau (v + K , w + K) & = \d \kappa (v,w) \, , & \mbox{for any $v,w \in \Gamma(K^\perp)$,} \label{eq-twist} \\
				\sigma (v +K , w+K) & = \frac{1}{2} \mathsterling_{k} g (v , w ) - \frac{1}{n} \epsilon \, g (v , w ) \, , & \mbox{for any $v, w \in \Gamma(K^\perp)$,} \label{eq-shear} 
			\end{align}
		\end{subequations}
		are referred to as the \emph{expansion}, the \emph{twist} and the \emph{shear} of $k$ respectively, and collectively, as the \emph{optical invariants} of $k$.

		  We define the \emph{expansion}, the \emph{twist} and the \emph{shear of $\mc{K}$} as the weighted sections $\breve{\epsilon} \in \Gamma(\mc{E}(1))$, $\breve{\tau} \in \Gamma(\bigwedge^2 H^*_K (1))$ and $\breve{\sigma} \in \Gamma( \odot^2_\circ H^*_K (1))$ respectively, whose trivialisations by a choice of generator of $\mc{K}$ are given by $\epsilon$, $\tau$ and $\sigma$ respectively, as given by \eqref{eq-ETS}.
	\end{defn}

	The terminology of Definition \ref{definition-tw-sh-exp} arises from the fact that the tensor fields $\epsilon$, $\tau$, and $\sigma$ defined by \eqref{eq-ETS} quantify the amount of expansion, twist and shear that a bundle of rays abreast to a given geodesic tangent to a generator of $\mc{K}$ undergoes as one moves along it. See e.g.\ \cite{Penrose1986} for details.
	
	\begin{prop}
		Let $(\mc{M},g,K)$ be an optical geometry with associated congruence of null geodesics $\mc{K}$. Let $\mathring{T}$ be the intrinsic torsion of the optical structure. Then
		\begin{enumerate}
			\item the \emph{expansion} of $\mc{K}$ is the image of the projection of $\mathring{T}$ to $\mc{G}_{-1}^0$;
			\item the \emph{twist} of $\mc{K}$ is the image of the projection of $\mathring{T}$ to $\mc{G}_{-1}^1$;
			\item the \emph{shear} of $\mc{K}$ is the image of the projection of $\mathring{T}$ to $\mc{G}_{-1}^2$.
		\end{enumerate}
	\end{prop}
	
	\begin{proof}
		This follows directly from Definition \ref{definition-tw-sh-exp} and the well-known formulae for the exterior derivative and the Lie derivative in terms of the Levi-Civita. Choose a section $k$ of $K$ and set $\kappa =g(k,\cdot)$. Then
		\begin{align*}
			\nabla_a \kappa_b & = (\d \kappa)_{a b} + \frac{1}{2} \mathsterling_k g_{a b} \, ,
		\end{align*}
		After contracting with $\delta^a_i \delta^b_j$, we immediately identify the first term with the twist of $k$. The trace of the second term with respect to $h^{i j}$ encodes the expansion, while its tracefree part, its shear.\end{proof}
	
	\begin{rem}
		In dimension six, the twist splits further into a self-dual part and an anti-self-dual part.
	\end{rem}
	
	As an immediate consequence of this proposition, we obtain:
	\begin{prop}
		Let $(\mc{M}, g,K)$ be an optical geometry with congruence of null curves $\mc{K}$ and intrinsic torsion $\mathring{T}$. Then
		\begin{enumerate}
			\item $\mc{K}$ is non-expanding geodetic if and only if the intrinsic torsion $\mathring{T} \in \Gamma(\slashed{\mc{G}}_{-1}^0)$;
			\item $\mc{K}$ is non-twisting geodetic if and only if the intrinsic torsion $\mathring{T} \in \Gamma(\slashed{\mc{G}}_{-1}^1)$;
			\item $\mc{K}$ is non-shearing geodetic if and only if the intrinsic torsion $\mathring{T} \in \Gamma(\slashed{\mc{G}}_{-1}^2)$.
		\end{enumerate}
	\end{prop}
	For clarity, we record these results in Table \ref{tab-geom-intors}.
	\begin{center}
		\begin{table}
			\begin{displaymath}
				{\renewcommand{\arraystretch}{1.5}
					\begin{array}{|c|c|c|}
						\hline
						\text{Congruence} & \text{Intrinsic torsion} \\
						\hline
						\hline
						\text{geodetic} & \slashed{\mc{G}}_{-2}^0  = \mc{G}^{-1} \\
						\hline
						\hline
						\text{non-expanding} & \slashed{\mc{G}}_{-1}^0  \\
						\text{non-twisting} & \slashed{\mc{G}}_{-1}^{1}  \\
						\text{non-shearing} & \slashed{\mc{G}}_{-1}^2 \\
						\hline
						\hline
						\text{parallel} & \slashed{\mc{G}}_0^0 = \{ 0 \} \\
						\hline
				\end{array}}
			\end{displaymath}
			\caption{\label{tab-geom-intors} Geometric properties and intrinsic torsion}
		\end{table}
	\end{center}

	\begin{rem}\label{rem-twist-reductions}
		A non-integrable optical geometry $(\mc{M},g,K)$ with congruence of null geodesics $\mc{K}$ can therefore fall into eight distinct classes according to whether the expansion, twist or shear vanishes or not. However, if non-vanishing, the twist provides an additional geometric structure on $\mc{M}$, and thus, a further reduction of the structure group of the frame bundle, and more specifically, of the screen bundle. Indeed, the stabiliser of the twist can be seen as a closed Lie subgroup of $\Sim^0(n)$, and more precisely, a reduction of the semi-simple part $\SO(n)$ of $\Sim^0(n)$. Further, since any section $\kappa$ of $\Ann(K^\perp)$ is preserved along the null geodesics of $\mc{K}$, so must $\d \kappa$ by the naturality of the exterior derivative. This means that the twist is always \emph{preserved} along the flow of $\mc{K}$, and in fact defines an additional geometric structure on the leaf space of $\mc{K}$. As we shall see, the r\^{o}le of the shear is quite different: it is the \emph{obstruction} to the conformal class of the screen bundle metric being preserved along $\mc{K}$. Thus, it is the absence of shear that defines an additional structure on $\mc{M}$, and more specifically, on the leaf space of $\mc{K}$.
	\end{rem}
	
	\subsubsection{Non-expanding congruences}
	The absence of expansion of a congruence of null geodesics will be relevant mostly in the context of non-shearing congruences. Here, we record only the following result, which follows from \eqref{eq-exp} and the fact that, for any $k \in \Gamma(K)$, $\mathsterling_k \left(k \hook \varepsilon \right) = \div k \, \left(k \hook \varepsilon \right)$.
	
	\begin{lem}
		Let $(\mc{M}, g,K)$ be an optical geometry with non-expanding congruence of null geodesics $\mc{K}$. Then for any $k \in \Gamma(K)$ such that the curves of $\mc{K}$ are affinely parametrised, we have
		\begin{align*}
			\mathsterling_k (k\hook \varepsilon) & = 0 \, .
		\end{align*}
		In particular, $k \hook \varepsilon$ descends to an orientation $(n+1)$-form on the leave space $(\ul{\mc{M}}, \ul{H})$ of $\mc{K}$.
	\end{lem}
	
	\subsubsection{Non-twisting congruences}\label{sec-non-twist-optical}
	Let $(\mc{M}, g,K)$ be an optical geometry with non-twisting congruence of null geodesics $\mc{K}$, i.e., the twist of $\mc{K}$ is identically zero. From the definition \eqref{eq-twist}, this can be equivalently expressed by the condition 
	\begin{align*}
		\kappa \wedge \d \kappa & = 0 \, , & \mbox{for any $\kappa \in \Gamma(\Ann(K^\perp))$.}
	\end{align*}
	A direct application of the Frobenius theorem gives immediately:
	\begin{prop}\label{prop-non-twisting}
		Let $(\mc{M}, g,K)$ be an optical geometry with congruence of null geodesics $\mc{K}$ and local leaf space $(\ul{\mc{M}},\ul{H})$. The following statements are equivalent.
		\begin{enumerate}
			\item $\mc{K}$ is non-twisting.
			\item $\mc{M}$ is locally foliated by a one-parameter family of $(n+1)$-dimensional submanifolds tangent to $K^\perp$, each containing an $n$-parameter family of null geodesics of $\mc{K}$.
			\item $\ul{H}$ is integrable, i.e., $\ul{\mc{M}}$ is locally foliated by $n$-dimensional submanifolds tangent to $\ul{H}$.
		\end{enumerate}
	\end{prop}
	
	\begin{rem}
		We note that for a non-twisting congruence of null geodesics $\mc{K}$, one can always find a local generator $k$ of $\mc{K}$ such that the $1$-form $\kappa = g(k,\cdot)$ is exact, i.e., $\kappa = \d u$ for some locally function on $\mc{M}$. The global existence of such a $1$-form is obstructed by the first De Rham cohomology group.
	\end{rem}

	\subsubsection{Non-shearing congruence}\label{sec-non-shear-optical}
	For any non-shearing congruence of null geodesics $\mc{K}$, we have, for any generator $k$ of $\mc{K}$,
	\begin{align*}
		\mathsterling_k g (v , w ) &= \frac{2}{n}\, \epsilon \, g (v , w ) \, ,  & \mbox{for any $v, w \in \Gamma(K^\perp)$,}
	\end{align*}
	where $\epsilon$ is the expansion of $k$. The geometric interpretation is then immediate.
	\begin{prop}\label{prop-non-shearing}
		Let $(\mc{M}, g,K)$ be an optical geometry with congruence of null geodesics $\mc{K}$ with leaf space $(\ul{\mc{M}},\ul{H})$. The following statements are equivalent.
		\begin{enumerate}
			\item $\mc{K}$ is non-shearing.
			\item The conformal class of the induced screen bundle metric is preserved along the geodesics of $\mc{K}$.
		\end{enumerate}
		If $\mc{K}$ is in addition non-expanding, then the induced metric on the screen bundle $H_K$ descends to a bundle metric on $\ul{H}$.
	\end{prop}
	
	The next result exhibits the existence of a family of linear connections compatible with an optical structure, and in particular, when the congruence is geodetic and non-shearing.
\begin{prop}\label{prop:lin_conn}
	Let $(\mc{M},g,K)$ be an optical geometry with congruence of null curves $\mc{K}$. Fix a splitting $(\ell^a , \delta^a_i , k^a)$, choose a $1$-form $f_a$, and denote by $\gamma_i$, $\epsilon h_{i j}$, $\tau_{i j} = \tau_{[i j]}$, $\sigma_{i j} = \sigma_{(i j)_\circ}$ and $E_i$ the components of $\nabla_{a} \kappa_{b}$ with respect to this splitting. Here $\kappa_a = g_{a b} k^{b}$ and $h_{i j}$ is the screen bundle metric.
	
	For each $t \in \R$, define a linear connection $\nabla^t$ by
	\begin{align*}
		\nabla^t_a \xi_b & = \nabla_a \xi_b - Q_{a b}{}^c \xi_c \, , & \mbox{for any $1$-form $\xi_a$,}
	\end{align*}
	where $Q_{a b c} = Q(t)_{a b c} = Q_{a b}{}^{d} g_{c d}$ is a tensor with components defined by
	\begin{align*}
		Q^{0 0 0} = Q_{0}{}^{0 0} = Q^{0}{}_{0}{}^0 = Q_{i}{}^{0 0} & = 0 \, , \\
		Q^0{}_{j}{}^0 = - Q^{0 0}{}_{j} & =  \gamma_j \, , &
		Q^0{}_{(j k)} & = -\frac{1}{n} f^0 h_{j k}  \, , \\
		Q_{i j}{}^0 = - Q_{i}{}^0{}_{j} & =  \frac{\epsilon}{n} h_{i j} + \tau_{i j} + \sigma_{i j}  \, , &
		Q_{i (j k)} & = -\frac{1}{n} f_i h_{j k}  \, , \\
		Q_{0 j}{}^0 = Q_{j 0}{}^0 = - Q^{0}{}_{0 k} = - Q_{0}{}^0{}_{j} & =  E_j \, , &
		Q_{0 (j k)} & = -\frac{1}{n} f_0 h_{j k}  \, , \\
		Q^{0}{}_{[j k]} & = t \tau_{j k} \, ,
	\end{align*}
	and all other components arbitrary.
	
	Then $\nabla^t$ is a connection compatible with $K$ and the conformal class $[h]$ on the screen bundle with torsion given by
	\begin{align*}
		T^{0}{}_{j}{}^{0} & = - \gamma_{j} \, , & T^{0}{}_{j}{}_{k} & =  \frac{1}{n} ( f^0 - \epsilon )h_{j k}  - \sigma_{j k}  \, ,  \\
		T^{0}{}_{0}{}^{0} & = 0  \, , & T^{0}{}_{0 k} & = 0 \, , \\
		T_{i j}{}^{0} & = - 2 \tau_{i j} \, , & T^{0}{}_{i j} = - T_{i}{}^{0}{}_{j} & = - \left( 1 + t \right) \tau_{i j} \, , \\
		T_{0 j}{}^{0} & = 0 \, , 
	\end{align*}
	and all other components arbitrarily determined by $-2Q_{[a b]}{}^c$.

As a special case, assuming that $\mc{K}$ is a non-shearing congruence of null geodesics, i.e.\
\begin{align*}
	(\nabla_{a} \kappa_{[b} ) \kappa_{c]} & = \frac{\epsilon}{n} \, h_{a [b} \kappa_{c]} + \tau_{a [b} \kappa_{c]} + \kappa_a E_{[b} \kappa_{c]} \, ,
\end{align*}
where $h_{a b} \delta^a_i \delta^b_j = h_{i j}$, $\tau_{a b} \delta^a_i \delta^b_j = \tau_{i j}$ and $E_{a} \delta^a_i = E_i$, we may define a family of linear connections $\nabla^t$ by choosing $Q_{a b c}$ to be
\begin{align*}
	Q_{a b c} & := \frac{\epsilon}{n} \left( h_{a b} \lambda_{c} - 2 h_{c (a} \lambda_{b)} \right) + 2 \tau_{a [b} \lambda_{c]} + t \, \lambda_a \tau_{b c} - 2 \kappa_{(a} \lambda_{b)} E_c + 2 \kappa_{(a} E_{b)} \lambda_{c} \, ,
\end{align*}
where $\lambda_{a} = \ell^{b} g_{a b}$. Then for each $t \in \R$, $\nabla^t$ is a linear connection that satisfies
\begin{align*}
	(\nabla^t_{a} \kappa_{[b} ) \kappa_{c]} & = 0 \, , &
	\nabla^t_a g_{b c} & = \frac{2}{n} \epsilon \lambda_a h_{b c} + 2 \, \lambda_a \kappa_{(b} E_{c)} - 2 \, E_a \kappa_{(b} \lambda_{c)} \, ,
\end{align*}
and has torsion given by
\begin{align*}
	T^t_{a b c} &  = - 2 \tau_{a b} \lambda_{c} - \left( 2 + 2 t \right) \lambda_{[a} \tau_{b] c} \, .
\end{align*}
In particular, the torsion satisfies
\begin{align*}
	T^{-2}_{[a b c]} & = 0 \, , & T^{-1}_{a (b c)} & = 0 \, .
\end{align*}
Further, for any $t \in \R$, $\nabla^t$ is torsion-free if and only if $\mc{K}$ is also non-twisting.
\end{prop}

\begin{proof}
It suffices to compute
\begin{align*}
	\nabla^t_a \kappa_b & = \nabla_a \kappa_b - Q_{a b}{}^c \kappa_c \, , &
	\nabla^t_a g_{b c} & = - 2 \, Q_{a (b c)} \, ,
\end{align*}
with the definitions given.
\end{proof}
	
	\begin{rem}
		The linear connections defined in the proposition above depend only on the choice of $k$ and $\lambda$. Note however that these are not the only connections preserving $K$ and the conformal structure on $H_K$.
	\end{rem}
	
	\subsection{Non-twisting non-shearing congruences}\label{sec-non-twist-non-shear-optical}

	These include two important classes of Lorentzian manifolds:
	\begin{enumerate}
		\item \emph{Robinson-Trautman spacetimes} for which the congruence is expanding, and
		\item \emph{Kundt spacetimes} for which the congruence is non-expanding.
	\end{enumerate}
	In other words, $(\mc{M},g,K)$ is a Robinson-Trautman spacetime if and only if the intrinsic torsion $\mathring{T}$ is a section of $\slashed{\mc{G}}_{-1}^1 \cap \slashed{\mc{G}}_{-1}^2$, but not a section of $\slashed{\mc{G}}_{-1}^0$,
	while a Kundt spacetime is one for which $\mathring{T}$ is a section of $\mc{G}^0 = \slashed{\mc{G}}_{-1}^0 \cap \slashed{\mc{G}}_{-1}^1 \cap \slashed{\mc{G}}_{-1}^2$.
	
	\begin{rem}
		In dimension three, the screen bundle of an optical gometry is one-dimensional, from which it follows that a congruence of null geodesics is necessarily non-twisting and non-shearing.
	\end{rem}
	
	Robinson-Trautman spacetimes were introduced in dimension four in \cite{Robinson1961/62}, and were later generalised to higher dimensions in \cite{Podolsky2006a, Ortaggio2008, Ortaggio2013, Ortaggio2015, Podolsky2015, Podolsky2016}.

	Kundt spacetimes were introduced in \cite{Kundt1961} in dimension four. Their higher-dimensional generalisations are investigated in the context of the algebraic properties of the Weyl tensor in  \cite{Podolsky2008, Krtous2012, Podolsky2013}. More recently, a $G$-structure approach to Kundt spacetimes is adopted in \cite{Aadne2019}.
	
	As we shall see in Section \ref{sec:conf-opt-str}, Robinson-Trautman spacetimes and Kundt spacetimes are conformally related, and for this reason, share important properties \cite{Podolsky2015}.
	
	\begin{prop}
		Let $(M,g,K)$ be an optical geometry with congruence of null curves $\mc{K}$. The following two statements are equivalent:
		\begin{enumerate}
			\item The distribution $K^\perp$ is integrable and the leaves tangent to $K^\perp$ are totally geodetic.
			\item $\mc{K}$ is a non-expanding, non-twisting and non-shearing congruence of null geodesics.
		\end{enumerate}
	\end{prop}
	
	\begin{proof}
		We must show $\nabla_v w \in \Gamma(K^\perp)$ for any $v, w \in \Gamma(K^\perp)$. But then, using our previous notation,
		\begin{align*}
			0 & = (\delta^a_i \nabla_a \delta^b_j) \kappa_b = - (\delta^a_i \nabla_a  \kappa_b) \delta^b_j  = \frac{1}{n} \epsilon h_{i j} + \tau_{i j} + \sigma_{i j} \, ,
		\end{align*}
		i.e., $\epsilon = \tau_{i j} = \sigma_{i j} = 0$.
	\end{proof}
	
	\subsubsection{Expanding non-twisting non-shearing congruences of null geodesics}
	In the neighbourhood of every point, there exists coordinates $(u,v,x^i)$ such that the Robinson--Trautman metric takes the form \cite{Robinson1961/62,Podolsky2006a}
	\begin{align*}
		g & = 2 \, \d u  \d v + 2 A_i \d x^i \d u+ B (\d u)^2 + \e^{2 \phi} \ul{h}_{i j} \d x^i  \d x^j \, ,
	\end{align*}
	where $v$ is an affine parameter along the congruence $\mc{K}$, $A_i$, $B$ are functions on $\mc{M}$, $\ul{h}_{i j}$ are functions on $\ul{\mc{M}}$, and $\phi$ is a function of $v$ only.
	
	\begin{exa}[Tangherlini-Schwarzschild metric]
		Let $(S^n,\ul{h})$ be an $n$-sphere $S^{n}$ with its round metric. Define the wrapped Lorentzian metric on $\mc{M} = \R \times \R^+ \times S^n$ with $(t,r) \in \R \times \R^+$
		\begin{align*}
			g & = - F(r) \d t^2 + F(r)^{-1} \d r^2 + r^2 \ul{h} \, , & \mbox{where} & &
			F(r) & = 1 - \frac{c}{r^{n-1}} \, ,
		\end{align*}
		and $c$ is a constant typically interpreted to be proportional to the mass of a static black hole. This is a Ricci-flat metric.
		
		The $1$-form
		\begin{align*}
			\kappa & = - \d t + F(r)^{-1} \d r \, , 
		\end{align*}
		annihilates the orthogonal complement of an optical structure $K$ with expanding non-twisting non-shearing congruence of null geodesics. Similarly, the $1$-form $\lambda =  \frac{1}{2} F(r) \d t + \frac{1}{2} \d r$ defines a second optical structure with expanding non-twisting non-shearing congruence of null geodesics.
		
		Other generalisations can be found   in \cite{Podolsky2006a} by replacing $S^n$ by some $n$-dimensional compact Einstein Riemannian manifold.
	\end{exa}
	
	\subsubsection{Non-expanding non-twisting non-shearing congruences of null geodesics}
	
	Combining Propositions \ref{prop-non-twisting} and \ref{prop-non-shearing} yields:
	\begin{cor}\label{proposition-SF->intconfH}
		Let $(\mc{M},g,K)$ be an optical geometry with  congruence of null geodesics $\mc{K}$. The following statements are equivalent.
		\begin{enumerate}
			\item $\mc{K}$ is non-expanding non-twisting non-shearing.
			\item the (local) leaf space $(\ul{\mc{M}},\ul{H},\ul{h})$ of $\mc{K}$ is foliated by $n$-dimensional submanifolds tangent to $\ul{H}$, each of which inherits a Riemannian structure from $h$.
		\end{enumerate}
	\end{cor}
	
	For an optical geometry $(\mc{M}, g , K)$ with non-expanding non-twisting non-shearing  congruence of null geodesics $\mc{K}$, in the neighbourhood of every point, there exists coordinates $(u,v,x^i)$ such that the metric takes the form
	\begin{align}\label{eq:Kundt_metric}
		g & = 2 \,  \d u \d v + 2 \, A_i \d x^i  \d u  + B (\d u)^2  +  \ul{h}_{i j} \d x^i \d x^j \, .
	\end{align}
	where $A_i = A_i (u,v,x^i)$, $B = B (u,v,x^i)$ and $\ul{h}_{i j} =  \ul{h}_{(i j)} (u,x^i)$.
	Here, $v$ is an affine parameter along the geodesics of $\mc{K}$, and $(u,x^i)$ are local coordinates on the leave space of $K$ with $u$ parametrising the leaves of $K^\perp$. Thus,
	\begin{align*}
		K & = \mathrm{span} \left( \frac{\partial}{\partial v} \right) \, , &
		K^\perp = \mathrm{span}\left( \frac{\partial}{\partial v} , \frac{\partial}{\partial x^i} \right) \, .
	\end{align*}
	The form of this metric remains the same under a change of coordinates $x^i = x^i (u, \wt{x})$. The $1$-form $\kappa = \d u$ annihilating $K^\perp$ remains unchanged, and we interpret the functions $\ul{h}_{i j}$ as the components of the screen bundle metric on $H_K$.
	
	\begin{rem}
		References \cite{Podolsky2013, Podolsky2015, Podolsky2016} give these special cases in terms of Weyl curvature conditions. Of interest is when the Weyl tensor is \emph{algebraically special} in the sense that $\kappa_{[a} W_{b c] f [d} k^f \kappa_{e]} = 0$,  where $k$ generates $\mc{K}$ and $\kappa = g(k,\cdot)$. In this case, the $v$-dependence of the functions $A_i$ can be integrated, i.e.\
		$A_i = \ul{A}_i^{(0)} + v \, \ul{A}_i^{(1)}$,
		where $\ul{A}_i^{(0)} = \ul{A}_i^{(0)} (u,x)$ and $\ul{A}_i^{(1)} = \ul{A}_i^{(1)} (u,x)$.
	\end{rem}

	As a direct corollary of Proposition \ref{prop:lin_conn}, we have:
	\begin{cor}\label{prop:lin_conn_Kundt}
		Let $(\mc{M}, g, K)$ be an optical geometry with non-expanding non-twisting non-shearing congruence $\mc{K}$ of null geodesics. Let $k^a$ be a generator of $\mc{K}$ for which the curves are affinely parametrised, and set $\kappa_a = g_{a b} k^b$. Choose a null $1$-form $\lambda_a$ such that $\lambda_a k^a =1$, so that
		\begin{align*}
			(\nabla_{a} \kappa_{[b} ) \kappa_{c]} & = \kappa_a E_{[b} \kappa_{c]}
		\end{align*}
		for some $1$-form $E_a$ annihilated by $K$.
		
		Define a linear connection
		\begin{align*}
			\nabla'_a v^b & = \nabla_a v^b + Q_{a c}{}^b v^c \, , &  \mbox{for any vector field $v^a$,}
		\end{align*}
		where
		\begin{align*}
			Q_{a b c} = Q_{a b}{}^{d} g_{d c} = - 2 \kappa_{(a} \lambda_{b)} E_c + 2 \kappa_{(a} E_{b)} \lambda_{c} \, .
		\end{align*}
		Then $\nabla'$ a torsion-free linear connection that satisfies
		\begin{align*}
			(\nabla'_{a} \kappa_{[b} ) \kappa_{c]} & = 0 \, ,  \\
			\nabla'_a g_{b c} & = 2 \, \lambda_a \kappa_{(b} E_{c)} - 2 \, E_a \kappa_{(b} \lambda_{c)} \, .
		\end{align*}
		In particular, $\nabla'$ preserves $K$ and the screen bundle metric.
	\end{cor}

	\begin{rem}
		The connection $\nabla'$ defined in Corollary \ref{prop:lin_conn_Kundt} is not the unique torsion-free linear connection compatible with $K$ and the screen bundle metric. Note also $\nabla'$ satisfies
		$\nabla'_{(a} g_{b c)} = 0$.
	\end{rem}
	
	\subsubsection{Integrable optical structures}\label{sec-parallel-optical}
	We finally arrive at the case where the optical geometry $(\mc{M},g,K)$ is integrable as a $G$-structure. This is equivalent to the Levi-Civita connection preserving the line distribution $K$. In other words, any non-vanishing section $k$ of $K$ is \emph{recurrent}, i.e.,
	\begin{align*}
		(\nabla_a k^{[b} ) k^{c]} & = 0 \, .
	\end{align*}
	This is a special case of Kundt spacetimes, and on the other hand the Lorentzian special case of \emph{Walker} manifolds \cite{walker50I}, see also \cite{walkerbook}, which are defined as pseudo-Riemannian manifolds with a distribution of parallel $p$-planes. In particular, the condition imposes a further condition on the local form of the metric: the functions $A_i = A_i^{(0)} (u,x)$  in formula \ref{eq:Kundt_metric}  are independent of $v$.
	
	Integrable optical structures, i.e., with a parallel null line distribution $K$, are crucial in the classification of holonomy groups of Lorentzian manifolds, see for instance, 
	\cite{bb-ike93}, \cite{bryant000}, \cite{Figueroa-OFarri00} and \cite{leistnerjdg}. Within the integrable optical structures there is a rich hierarchy. On the one hand there are reductions of the screen bundle to a subgroup $H$  of $\SO^0(n)$. Based on results in \cite{bb-ike93}, it was shown in  \cite{leistnerjdg} that $H$ must be a Riemannian holonomy group, a result that yields a classification of Lorentzian holonomy groups. This also provides a classification of Lorentzian manifolds with parallel spinors, see also the survey \cite{galaev-leistner-esi}. 
	Global aspects of Lorentzian manifolds with special holonomy are discussed in \cite{Laerz2010}, \cite{Baum2014}, \cite{leistner-schliebner13} and \cite{Schliebner15}.
	
	On the other hand there is the reduction to the stabiliser $\SO(n)\ltimes \R^n$ of a null vector, which implies the existence of a parallel section of $K$, that is, a parallel null vector field. In this case, not only the functions $A_i = \ul{A}_i^{(0)} (u,x)$ in  \eqref{eq:Kundt_metric}  are independent of $v$, but also the function $B=\ul{B}(u,x)$. Explicitly, the metric is given by
	\begin{align}\label{pnvf}
		g & = 2 \,  \d u\, \d v + 2  \, \ul{A}_i (u,x) \d x^i\d u + \ul{B}(u,x) (\d u)^2 +  \ul{h}_{i j}(u,x)  \d x^i\,\d x^i  \, , 
	\end{align}
	and $\partial_v$ is a parallel null vector field. Sometimes these manifolds are called {\em Brinkmann waves} or  {\em Brinkmann spaces}. The spaces discovered by Brinkmann   in  \cite{brinkmann24,brinkmann25} in the context of conformal geometry do indeed have a parallel vector null field, however they are of dimension four and in addition Ricci-flat. Consequently their holonomy reduces to $\R^2\subset \SO(2)\ltimes \R^2$ and the metric is of the form 
	\[g=2\,\d u\,\d v +B(u,x,y) (\d u)^2+\d x^2+\d y^2,\] 
	so a very special case of the metrics in \eqref{pnvf}. The metrics found by Brinkmann, without the Ricci-flat condition,   generalise to higher dimensions to the so called {\em pp-waves}, which stands for `plane-fronted  with parallel rays' (that is $\ul{h}_{i j}=\delta_{ij}$  and $\partial_v$ parallel). For pp-waves   the holonomy reduces to the abelian ideal  $\R^n$ in $\SO(n)\ltimes \R^n$. This reduction is equivalent to the existence of a parallel section of $K$ and the flatness of the screen bundle $H_K\to \mathcal M$, or equivalently to the curvature condition that $R(X,Y)U\in K$,  for all $U\in K^\perp$ and $X,Y\in TM$, see \cite{leistner05c}. The local form of the metric \eqref{eq:Kundt_metric} for pp-waves becomes
	\begin{align}\label{pp}
		g & = 2 \,  \d u\,  \d v   + B(u,x)  (\d u)^2 +  \delta_{i j} \d x^i \, \d x^j \, .
	\end{align}
	Their Ricci tensor is  given by $\Delta B \,(\d u)^2$, where $\Delta$ is the Laplacian for the $x^i$ coordinates, so the vacuum Einstein equation simplifies to the Euclidean wave equation.
	These metrics have the interesting property that all their scalar invariants vanish. This was first observed by Penrose and established in higher dimension for example in \cite{Coley2004a,Coley2006}. 
	Within the pp-waves there are the so-called {\em plane waves} for which 
	\[ \ul{B}(u,x)=x^i \ul{Q}_{ij}(u)x^j\] 
	in the metric (\ref{pp}), with a $u$-dependent symmetric matrix $\ul{Q}_{ij}(u)$. Plane waves appear as Penrose-limits \cite{penrose76}, see also \cite{CveticLuPope02,BlauFigueroa-OFarriHullPapadopoulos02}, and as homogeneous supersymmetric M-theory backgrounds \cite{Figueroa-OFarriMeessenPhilip05}. Conformal aspects of pp-waves, in particular their ambient metrics are studied in \cite{leistner-nurowski08}.
	
	Homogeneous  plane waves and pp-waves have been studied and classified in \cite{blau-oloughlin03} and \cite{GlobkeLeistner16}. 
	The geodesic completeness of {\em compact} pp-waves is proved in \cite{leistner-schliebner13}.

	A subclass within the plane waves are the Lorentzian locally symmetric spaces with {\em non-constant sectional} curvature. They are given by
	\begin{align*}
		g & = 2 \,  \d u\,  \d v   +(x^i Q_{ij}x^j) (\d u)^2 +  \delta_{i j} \d x^i \, \d x^j \, ,
	\end{align*}
	where $Q_{ij}$ is a constant symmetric matrix.
	On   $\mc{M}=\R^{n+2}$ this defines  a  globally symmetric space with
	solvable transvection group, the so-called  {\em Cahen-Wallach spaces} \cite{cahen-wallach70}. 
	Their compact quotients are described in \cite{KathOlbrich15}.
	
	Note that all these subclasses of pp-waves are {\em} not distinguished  by a further reduction of the holonomy group but rather by conditions on the derivative of the curvature. If $\ul{Q}_{ij}$ is non-degenerate, for all of them the connected component of the holonomy group is equal to  $\R^n$ and hence are indecomposable. Also there are many more  examples of integrable optical structures that generalise the pp-waves and have been considered in the literature, for example the {\em plane fronted waves} \cite{candela-flores-sanchez03},  for which the $ \ul{h}_{i j}$ in (\ref{pnvf}) do not depend on $u$, i.e.,  $ \ul{h}_{i j}$ is just a Riemannian metric of the $x^i$ coordinates.  Examples of earlier results are in \cite{schimming74,takeno62} and of course there is a plethora of general relativity and physics literature about them.

	\subsection{Non-twisting shearing congruences of null geodesics}
	We present an example of an optical geometry with non-twisting shearing congruences of null geodesics.
	
	\begin{exa}[The black ring in dimension five]
		Using the coordinates $(t, x, y, \phi, \psi)$ given in \cite{Elvang2003,Pravda2005}, the five-dimensional black ring discovered in \cite{Emparan2002} is described by the metric
		\begin{multline*}
			g = - \frac{F(x)}{F(y)} \left( \d t + R \sqrt{\lambda \nu} (1+y) \d \psi \right)^2 + \\
			+ \frac{R^2}{(x-y)^2} \left( - F(x) \left( G(y) \d \psi^2 + \frac{F(y)}{G(y)} \d y^2 \right) + F(y)^2 \left( \frac{\d x^2}{G(x)} + \frac{G(x)}{F(x)} \d \phi^2 \right) \right),
		\end{multline*}
		where
		\begin{align*}
			F(\xi) & := 1 - \lambda \xi \, , & G(\xi) & := (1 - \xi^2)(1-\nu \xi) \, .
		\end{align*}
		Here, $R$, $\lambda$ and $\mu$ are positive constants with $\lambda, \nu <1$.
		There are a number of possible ranges for the coordinates to ensure that the metric is of Lorentzian signature. For specificity, we restrict ourselves only to the region
		\begin{align*}
			\left\{ (x,y,\phi,\psi,t) : -1 < x < 1 \, , \tfrac{1}{\lambda} < y < \tfrac{1}{\nu} \right\} \, .
		\end{align*}
		Then, following \cite{Taghavi-Chabert2011}, the $1$-forms 
		\begin{align*}
			\kappa & := \frac{R \sqrt{-F(x) G(y)}}{\sqrt{2} (x-y)} \left(\frac{\sqrt{F(y)}}{G(y)} \d y + \i  \d \psi \right), &
			\lambda & := \frac{R \sqrt{-F(x) G(y)}}{\sqrt{2} (x-y)} \left( \frac{\sqrt{F(y)}}{G(y)} \d y - \i  \d \psi \right)
		\end{align*}
		are real and define two optical structures $K$ and $L$ with expanding non-twisting shearing congruences of null geodesics.
	\end{exa}

	\subsection{Twisting congruences of null geodesics}
	As mentioned in Remark \ref{rem-twist-reductions}, the twist of an optical geometry with null geodesic congruence, if non-vanishing, provides an additional geometrical structure in its own right not only on the screen bundle, but also on the leaf space of the congruence. To determine which structures arise from the twist, we need to examine its algebraic properties.
	
	\begin{defn}
		Let $(\mc{M},g,K)$ be an optical geometry with congruence of null geodesics $\mc{K}$. The \emph{rank} of the twist of $\mc{K}$ is the rank of any section $\kappa \in \Gamma(\Ann(K^\perp))$, i.e., the positive integer $d$ such that $\kappa \wedge \left( \d \kappa \right)^d \neq 0$ and $\kappa \wedge \left( \d \kappa \right)^{d+1} = 0$.
	\end{defn}
	Clearly, $\mc{K}$ is non-twisting if and only if the twist has rank zero. If $\ul{\kappa}$ is the $1$-form on $\ul{\mc{M}}$ induced from $\kappa$, then $\ul{\kappa}$ has the same rank as $\kappa$.
	
	\subsubsection{Maximally twisting congruences}
	We now focus on the case $d=m$, where we assume $n=2m$ or $n=2m+1$.
	
	\begin{defn}
		Let $(\mc{M},g,K)$ be an optical geometry of dimension $2m$ or $2m+1$ with congruence of null geodesics  $\mc{K}$. We say that $\mc{K}$ is \emph{maximally twisting} if its twist has rank $m$, i.e., any section $\kappa$ of $\Ann(K^\perp)$ has maximal rank, i.e.\
		$\kappa \wedge (\d \kappa )^m \neq 0$ and $\kappa \wedge (\d \kappa )^{m+1} = 0$.
	\end{defn}
	Note that in dimensions $2m+2$, the twist then defines an orientation on the screen bundle, and hence on the leaf space. If $m$ is odd, this also singles out a direction for sections of $K$ (but not when $m$ is even.) In fact, as an alternative normalisation to Proposition \ref{prop-nor-twist}, one can pick a section of $K$ in such a way that $\kappa \wedge (\d \kappa)^m = k \hook \varepsilon$.
	
	\begin{rem}
		In dimensions four and five, a twisting congruence is necessarily maximally twisting.
	\end{rem}
	
	The next proposition is particularly important: it tells us that for a maximally twisting congruence, one can associate a distinguished splitting of the tangent bundle. This means in particular that the structure group of the frame bundle is reduced from $\Sim^0(n)$ to $\CO^0(n)$ --- in fact, the stabiliser of the twist yields a further reduction to a subgroup of $\CO^0(n)$.
	\begin{prop}\label{prop-adapted_frame_max_tw}
		Let $(\mc{M}, g, K)$ be $(2m+2)$-dimensional optical geometry with maximally twisting geodetic congruence $\mc{K}$. Then there exists a generator $k$ of $\mc{K}$ and a null vector field $\ell$ with $g(k, \ell) = 1$ such that   $\kappa = g( k , \cdot)$ satisfies
		\begin{align*}
			\d \kappa (k , \cdot) & = 0 \, , & \d \kappa (\ell , \cdot) & = 0 \, .
		\end{align*}
		In particular, the geodesics of $\mc{K}$ are affinely parametrised by $k$, and we can write
		\begin{align}
			\nabla_{[a} \kappa_{b]} & = \tau_{a b} \, ,   & \tau_{a b} = \tau_{i j} \delta^i_a \delta^j_b \, , \label{eq-can-tau} 
		\end{align}
		where $\tau_{i j}$ is the twist of $k$. The pair $(k^a , \ell^a)$ is unique up to boosts constant along $K$.
	\end{prop}
	
	\begin{proof}
		We can assume that there exists a generator $k$ of $K$ such that the geodesics are affinely parametrised. Then $\kappa = g(k, \cdot)$ takes the form
		\begin{align*}
			\nabla_{[a} \kappa_{b]} = \tau_{a b} + 2 \, \kappa_{[a} \beta_{b]} \, ,
		\end{align*}
		where $\tau_{a b}=\tau_{i j} \delta^i_a \delta^j_b$ represents the twist, and $\beta_{a}k^a=0$. We seek $\wt{\ell}^a = \ell^a - \phi^i \delta_i^a - \frac{1}{2} \phi^i \phi_i k^a$, where $g(k,\ell)=1$,  such that $\d \kappa (\wt{\ell}, \cdot) = 0$. We find
		\begin{align*}
			(- \phi^i \tau_{i j} + \beta_j )\delta^j_b + \phi^i \beta_i \kappa_b & = 0 \, .
		\end{align*}
		Since $\tau_{i j}$ is non-degenerate, this equation has unique solution  $\phi^i = (\tau^{-1})^{i j} \beta_j$.
		
		Thus we have associated to $k^a$ a unique null vector field $\ell^a$. There remains the freedom of changing the pair $(k, \ell)$ by means of a boost \eqref{eq-boost} constant along $K$ so as to preserve the affine parametrisation of the geodesics with respect to $k$.
	\end{proof}
	
	We now turn to the interpretation on the leaf space. Let us first recall some notions. A \emph{contact distribution} or \emph{contact structure} on a $(2m+1)$-dimensional smooth manifold $\ul{\mc{M}}$ is a rank-$2m$ distribution $\ul{H}$ such that any non-vanishing section $\underline{\kappa}$ of $\Ann(\ul{H})$, referred to as a \emph{contact $1$-form}, has maximal rank, i.e., $\underline{\kappa} \wedge \left(  \d \underline{\kappa} \right)^{m} \neq 0$. We refer to the pair $(\ul{\mc{M}}, \ul{H})$ as a \emph{contact manifold}.
	Every choice of contact $1$-form $\underline{\kappa}$ in $\Ann(\ul{H})$ defines a canonical splitting of the exact sequence $0 \longrightarrow \ul{H}  \longrightarrow  T \ul{\mc{M}} \longrightarrow  T \ul{\mc{M}} / \ul{H} \longrightarrow  0$ by means of the \emph{Reeb vector field}, the unique vector field $\underline{\ell}$ that satisfies $\underline{\kappa} (\underline{\ell}) = 1$ and $\d \underline{\kappa} (\underline{\ell}, \cdot) = 0$.
	
	\begin{prop}\label{proposition-maxTw-contact}
		Let $(\mc{M}, g,K)$ be a $(2m+2)$-dimensional optical geometry with congruence of null geodesics $\mc{K}$ and leaf space $(\ul{\mc{M}},\ul{H})$. The following statements are equivalent.
		\begin{enumerate}
			\item $\mc{K}$ is maximally twisting. \label{item:K-max-tw}
			\item $\ul{H}$ is a contact distribution. \label{item:contact}
		\end{enumerate}
		
		Further,  for each pair $(k,\ell)$ where $k$ generates affinely parametrised geodesics of $\mc{K}$ and $\ell$ is a null vector field with $g(k, \ell) = 1$, the $1$-form $\kappa = g( k , \cdot)$ satisfies
		\begin{align*}
			\d \kappa (k , \cdot) & = 0 \, , & \d \kappa (\ell , \cdot) & = 0 \, ,
		\end{align*}
		$\kappa$ descends to a contact form on $(\ul{\mc{M}}, \ul{H})$ and $\ell + K^\perp$ to its corresponding Reeb vector field.
	\end{prop}
	
	\begin{proof}
		The equivalence between statements \eqref{item:K-max-tw} and \eqref{item:contact} follows immediately from Lemma \ref{lem-proposition-geod}, while the existence of the pair $(k,\ell)$ with the properties stated in the proposition follows from Proposition \ref{prop-adapted_frame_max_tw}.
	\end{proof}
	
	Dealing with the odd-dimensional case, we have the following proposition:
	\begin{prop}
		Let $(\mc{M},g,K)$ be a $(2m+3)$-dimensional optical geometry with maximally twisting congruence of null geodesics  $\mc{K}$ and leaf space $(\ul{\mc{M}}, \ul{H})$. Then for any non-vanishing section $\ul{\kappa}$ of $\Ann(\ul{H})$, the distribution $\ul{H}' :=\ker \d \underline{\kappa} \cap \ker \underline{\kappa}$ is a line subbundle of $\ul{H}$, and $\ul{H}/\ul{H}'$ descends to a contact distribution on the $(2m+1)$-dimensional leaf space of $\ul{H}'$.
	\end{prop}
	
	\begin{proof}
		The $1$-form has rank $m$ by assumption, and it easily follows that $\ul{H}' := \ker \d \ul{\kappa} \cap \ker \ul{\kappa}$ must be one-dimensional. Let $u$ be a section of $\ul{H}'$. Then clearly $\mathsterling_u \ul{\kappa} = 0$ and $\mathsterling_u \d \ul{\kappa} = 0$, i.e. $\ul{\kappa}$ descends to a $1$-form of rank $m$ on the $(2m+1)$-dimensional leaf space of $\ul{H}'$, and thus annihilates a contact distribution there.
	\end{proof}
	
	\begin{exa}[The Kerr metric in dimension four]\label{exa:Kerr4}
		The \emph{Kerr metric} describes a rotating black hole with mass $M$ and angular momentum $a$. In local coordinates $(r, u, \theta, \phi)$, it is given by \cite{Kerr1963} 
		\begin{multline*}
			g = \left(r^2 + a^2 \cos^2 \theta \right) \left( (\d \theta)^2 + \sin^2 \theta (\d \phi)^2 \right) + 2 \left( \d u + a \sin^2 \theta \d \phi \right) \left( \d r + a \sin^2 \theta \d \phi \right) \\
			- \left( 1 - \frac{2 M r}{r^2 + a^2 \cos^2 \theta} \right) \left( \d u + a \sin^2 \theta \d \phi \right)^2 \, .
		\end{multline*}
		The $1$-form
		\begin{align*}
			\kappa & = \d u + a \sin^2 \theta \d \phi \, ,
		\end{align*}
		annihilates the orthogonal complement of an optical structures $K$ with expanding twisting non-shearing congruences of null geodesics -- it is generated by the null vector field $k = g^{-1} (\kappa,\cdot) = \frac{\partial}{\partial r}$. There is a second optical structure with expanding twisting non-shearing congruences of null geodesics defined by the $1$-form
		\begin{align*}
			\lambda & = \d r + a \sin^2 \theta \d \phi  - \frac{1}{2} \left( 1 - \frac{2 M r}{r^2 + a^2 \cos^2 \theta} \right) \kappa \, .
		\end{align*}
		
		More generally, the \emph{Pleba\'{n}ski-Demia\'{n}ski metric} \cite{Plebanski1976} is a solution to the Einstein-Maxwell equations depending on seven parameters, and which contains the Kerr metric as a limiting case. It also admits two optical structures with expanding twisting non-shearing congruences of null geodesics.
	\end{exa}

	\begin{exa}[The Myers-Perry metric in dimension higher than four]
		The Kerr solution was generalised to higher dimensions by Myers and Perry \cite{Myers1986}. In dimension $2m$, it describes a rotating black hole with mass $M$ and $m$ rotation parameters $a_\alpha$. In local coordinates $(r,u,\mu_0, \mu_\alpha, \phi_\alpha)$,  where $\mu_0^2 + \sum_{\alpha =1}^m \mu_\alpha^2 = 1$, the metric takes the form
		\begin{multline*}
			g = - ( \d u )^2 + 2 \left( \d u + \sum_{\alpha=1}^m a_\alpha \mu_\alpha^2 \d \phi_\alpha \right) \d r  \\
			+ r^2 (\d \mu_0)^2 + \sum_{\alpha=1}^m \left(r^2 + a_\alpha^2\right) \left( ( \d \mu_\alpha )^2  + \mu_\alpha^2 (\d \phi_\alpha)^2 \right) 
			+ \frac{M r^2}{P F} \left( \d u + \sum_{\alpha=1}^m a_\alpha \mu_\alpha^2 \d \phi_\alpha \right)^2 \, ,
		\end{multline*}
		where
		\begin{align*}
			F & = 1 - \sum_{\alpha=1}^m \frac{a_\alpha^2 \mu_\alpha^2}{r^2 + a_\alpha^2} \, , &
			P & = \prod_{\alpha=1}^m (r^2 + a_\alpha^2).
		\end{align*}
		The $1$-form
		\begin{align*}
			\kappa & = \d u + \sum_{\alpha=1}^m a_\alpha \mu_\alpha^2 \d \phi_\alpha 
		\end{align*}
		annihilates the orthogonal complement of an optical structures $K$ with expanding twisting congruences of null geodesics that is shearing when $n>2$. It is also maximally twisting as can be seen from
		\begin{align*}
			\d \kappa & = 2 \sum_{\alpha=1}^m a_\alpha \mu_\alpha \d \mu_\alpha \wedge \d \phi_\alpha\, ,
		\end{align*}
		which is non-zero whenever $\mu_\alpha$ is non-vanishing for any $\alpha = 1, \ldots , m$, i.e., $\mc{K}$ is maximally twisting.
		
		More generally, the \emph{Kerr-NUT-(A)dS metric} presented in \cite{Chen2006} is an Einstein metric in dimension $2m+2$ depending on a cosmological constant, mass parameter, $2m-1$ NUT parameters, and $m$ rotation parameters. It also admits two expanding maximally twisting congruences of null geodesics. These are shearing for all $m>1$.
		
		Similar results hold in odd dimensions.
	\end{exa}
	
	\subsubsection{Twisting non-shearing congruences}
	Under the non-shearing assumption, one can obtain stronger results. In particular, one can \emph{single out} a generator of the congruence by normalising its twist. This should be contrasted with the situation in Hermitian geometry, where the hermitian form has a fixed norm.
	
	\begin{prop}\label{prop-nor-twist}
		Let $(\mc{M},g,K)$ be an optical geometry with twisting non-expanding non-shearing congruence of null geodesics $\mc{K}$. Then there exists a unique generator $k$ of $\mc{K}$ such that the geodesics are  affinely parametrised with respect to $k$ and its twist $\tau$ satisfies $\tau_{i j} \tau^{i j} = 2 r$,  where $r$ is the rank of $\tau$.
	\end{prop}
	\begin{proof}
		Let $k$ be a section of $K$ generating affinely parameterised geodesics with twist $\tau_{i j}$. Then we can write
		\begin{align*}
			( \d \kappa)_{a b} & = \tau_{a b} + 2 \kappa_{[a} \beta_{b]} \, ,
		\end{align*}
		where $\tau_{a b} = \tau_{i j} \delta^i_a \delta^j_b$ and $k^b \beta_b =0$. In particular, $\| \d \kappa \| = \| \tau \|$. Set $\wt{k}^a = \sqrt{2r} \| \tau \|^{-1} k^a$. Then, with $\wt{\kappa} = g(\wt{k},\cdot)$, we have
		\begin{align*}
			\d \wt{\kappa} &=  -\sqrt{2r} \| \tau \|^{-2} ( \d \| \tau \| ) \wedge \kappa + \sqrt{2r} \| \tau \|^{-1} \d \kappa \, ,
		\end{align*}
		so that
		\begin{align*}
			\| \wt{\tau} \|^2 & =  \| \d \wt{\kappa} \|^2 = 2r \| \tau \|^{-2} \| \d \kappa \|^2 = 2r \, .
		\end{align*}
		Since $\mc{K}$ is non-expanding and non-shearing, we have $\mathsterling_k \| \tau \| = 0$. This means that $\wt{k}$ generates affinely parameterised geodesics.
		
		Finally, it is straightforward to check that any other generator of $\mc{K}$ satisfying these properties must be either $\wt{k}$ itself, or $-\wt{k}$. Since we assume that the congruence is oriented, we obtain uniqueness.
	\end{proof}
	
	\vspace{2.5mm}

	Combining with Proposition \ref{prop-adapted_frame_max_tw} and Proposition \ref{prop-nor-twist} yields
	\begin{prop}\label{prop-adapted_frame_max_tw-nsh-nexp}
		Let $(\mc{M}, g, K)$ be a $(2m+2)$-dimensional optical geometry with maximally twisting non-expanding non-shearing congruence of null geodesics $\mc{K}$. Then there exists a generator $k$ of $\mc{K}$ and a null vector field $\ell$ such that $g(k, \ell) = 1$ and  $\kappa = g( k , \cdot)$ satisfies
		\begin{align*}
			\d \kappa (k , \cdot) & = 0 \, , & \d \kappa (\ell , \cdot) & = 0 \, ,
		\end{align*}
		and the twist $\tau$ of $k$ satisfies $\tau_{i j} \tau^{i j} = 2m$. The pair $(k^a , \ell^a)$ is unique.
	\end{prop}

	The twist $\breve{\tau}$ of a congruence of null geodesics $\mc{K}$ also induces a bundle endomorphism of $H_K$ by composing $\breve{\tau}$ with the inverse metric $h^{-1}$ on $H_K$.
	
	\begin{defn}
		Let $(\mc{M}, g, K)$ be an optical geometry with congruence of null geodesics $\mc{K}$. The \emph{twist endomorphism} of $\mc{K}$ is the section $\breve{F}$ of $\mathrm{End}(H_K)(1)$ defined by
		\begin{align}\label{eq-twist-endo}
			\breve{F} & := h^{-1} \circ \breve{\tau} \, ,
		\end{align}
		where $h$ is the bundle metric on $H_K$ and $\breve{\tau}$ the twist of $\mc{K}$. Its trivialisation by some generator $k$ of $\mc{K}$ will be called the \emph{twist endomorphism associated to $k$}.
	\end{defn}
	Note that if $\kappa$ has rank $d$, its associated bundle endomorphism $F$ has matrix rank $2d$.

	\begin{prop}\label{prop-twist-endo}
		Let $(\mc{M},g,K)$ be an optical geometry with non-expanding non-shearing congruence of null geodesics $\mc{K}$ with leaf space $(\ul{\mc{M}}, \ul{H})$. Let $k$ be a generator of affinely parametrised geodesics of $\mc{K}$. Then the twist endomorphism $F$ associated to $k$ descends to an endomorphism of $\ul{H}$ on $\ul{\mc{M}}$.
	\end{prop}
	
	\begin{proof}
		This follows immediately from Propositions \ref{prop-non-shearing}, Remark \ref{rem-twist-reductions} and the definition of twist endomorphism.
	\end{proof}

	A \emph{contact sub-Riemannian structure} on a $(2m+1)$-dimensional smooth manifold $\ul{\mc{M}}$ consists of a contact distribution $\ul{H}$ equipped with a bundle metric $\ul{h}$.
	
	Now, putting Propositions \ref{proposition-maxTw-contact} and \ref{proposition-SF->intconfH} together proves:
	\begin{cor}\label{cor:sub-Riem}
		Let $(\mc{M},g,K)$ be a $(2m+2)$-dimensional optical geometry with congruence of null geodesics $\mc{K}$. The following two statements are equivalent:
		\begin{enumerate}
			\item $\mc{K}$ is maximally twisting non-shearing.
			\item The leaf space $(\ul{\mc{M}},\ul{H},\ul{h})$ is a contact sub-Riemannian manifold.
		\end{enumerate}
	\end{cor}
	Analogous results were proved in \cite{Alekseevsky2018,Alekseevsky2021}.

	\begin{exa}[The (A)dS-Taub-NUT metric \cite{Taub1951,Newman1963,Awad2002,Alekseevsky2021,TaghaviChabert2022}]
		Let $\ul{\mc{M}}$ be a circle bundle over a $2m$-dimensional K\"{a}hler-Einstein manifold with metric $\ul{h}$ and K\"{a}hler form $\ul{\omega}$ and non-zero Ricci scalar $4m  \underline{\Lambda}$. Choose a local $1$-form $\ul{A}$ such that $\d \ul{A} = \ul{\omega}$. Denote by $t$ the fiber coordinate of $\ul{\mc{M}}$, and let $\ul{\alpha} = \d t  + \ul{A}$. The \emph{(A)dS-Taub-NUT metric} is the Einstein metric with Ricci scalar $2(m+1)\Lambda$ defined on the radial extension $\R^+ \times \ul{\mc{M}}$ that is given by
		\begin{align*}
			g & = - F(r) (\ul{\alpha})^2  + F(r)^{-1} (\d r)^2  + \ul{h}, 
		\end{align*}
		where $F(r)$ is a smooth function that depends on $\Lambda$, $\underline{\Lambda}$ and a third constant $M$, and satisfies the differential equation
		\begin{align*}
			\frac{\d}{\d r}\left( \frac{(r^2 +  \ul{\Lambda}^2)^{m}}{r} F(r)\right)
			& =  \frac{(r^2 + \ul{\Lambda}^2)^{m}}{r^2}  \ul{\Lambda} -  \frac{(r^2 +  \ul{\Lambda}^2)^{m+1}}{r^2}  \frac{\Lambda}{\ul{\Lambda}^2} \, .
		\end{align*}
		Here, the `mass' parameter $M$ arises as the constant of integration.
		
		The $1$-form
		\begin{align*}
			\kappa & = \ul{\alpha} + F(r)^{-1} \d r \, ,
		\end{align*}
		annihilates the orthogonal complement of an optical structure $K$ with expanding twisting non-shearing congruence of null geodesics. In fact, the congruence is maximally twisting since $\d \kappa = \ul{\omega}$.
		
		A second maximally twisting non-shearing congruence of null geodesics can be seen by considering the $1$-form $\lambda = - \ul{\alpha} + F(r)^{-1} \d r$.
	\end{exa}

	\section{Conformal optical geometry}\label{sec:conf-opt-str}
	
	\subsection{Conformal optical structures}
	Most of the geometric properties of optical structures turn out to be conformally invariant, and for this reason, we extend their definitions to the conformal setting. Here, we follow the treatment of \cite{Bailey1994}. Recall that a \emph{conformal manifold} consists of a pair $(\mc{M}, \mbf{c})$,  where $\mc{M}$ is an $(n+2)$-dimensional smooth manifold and $\mbf{c}$ a conformal structure on $\mc{M}$, that is, an equivalence class of metrics on $\mc{M}$, every pair of which differ by a conformal factor, i.e., two metrics $g$ and $\widehat{g}$ in $\mbf{c}$ are related via
	\begin{align}\label{eq-conf-res}
		\widehat{g} & = \e^{2 \, \varphi} g \, , & \mbox{for some smooth function $\varphi$ on $\mc{M}$.}
	\end{align}
	The respective Levi-Civita connections $\nabla$ and $\hat{\nabla}$ of $g$ and $\widehat{g}$ are then related by
	\begin{align}\label{eq-LC-transf}
		\begin{aligned}
			\wh{\nabla}_a v^b & = \nabla_a v^b + \Upsilon_a v^b - v_a \Upsilon^b + \delta_a^b \Upsilon_c v^c \, , & \mbox{for any $v^a \in \Gamma( T \mc{M})$,} \\
			\wh{\nabla}_a \alpha_b & = \nabla_a \alpha_b - \Upsilon_a \alpha_b - \Upsilon_b \alpha_a + g_{a b} \Upsilon^c \alpha_c \, , & \mbox{for any $\alpha_a \in \Gamma( T^* \mc{M})$.} 
		\end{aligned}
	\end{align}
For each $w \in \mbf{R}$, there are associated density bundles denoted by $\mc{E}[w]$, and sections thereof are densities of \emph{conformal weight $w$}. The conformal structure $\mbf{c}$ can equivalently be encoded in terms of a global non-degenerate section  $\bm{g}$ of $\odot^2 T^* \mc{M} \otimes \mc{E}[2]$ of Lorentzian signature, referred to as the \emph{conformal metric}. The \emph{bundle of conformal scales} is a choice of ray subbundle $\mc{E}_+[1]$ of $\mc{E}[1]$: sections thereof are in one-to-one correspondence with metrics in $\mbf{c}$: any section $s$ of $\mc{E}_+[1]$ defines a metric in $\mbf{c}$ by $g = s^{-2} \bm{g}$. The Levi-Civita connection extends to a linear connection on any of $\mc{E}[w]$. One can check that, in our previous notation,
	\begin{align*}
		\wh{\nabla}_a s & = \nabla_a s + w \Upsilon_a s \, .
	\end{align*}
	We shall denote the covariant exterior derivative $\d^{\nabla}$, i.e. for any (weighted) $p$-form $\alpha$,
	\begin{align*}
		(\d^{\nabla} \alpha)_{a_0 a_1 \ldots a_p} & := \nabla_{[a_0} \alpha_{a_1 \ldots a_p]} \, .
	\end{align*}
	We shall use $\bm{g}$ to identify sections of $T \mc{M}$ with those of $T^* \mc{M} \otimes \mc{E}[2]$. The orientation on $\mc{M}$ yields a global section $\bm{\varepsilon}$ of $\bigwedge^{n+2} T^*\mc{M} \otimes \mc{E} [n+2]$, the \emph{weighted volume form} with the property that for any $s \in \mc{E}_+[1]$, $\varepsilon = s^{-n-2} \bm{\varepsilon}$ is the volume form of the metric $g = s^{-2} \bm{g}$.

	\begin{defn}
		Let $(\mc{M}, \mbf{c})$ be an oriented and time-oriented Lorentzian conformal manifold.  An \emph{optical structure} on $(\mc{M},\mbf{c})$ is given by a vector distribution $K\subset T \mc{M}$ of tangent null lines. We refer to $(\mc{M},\mbf{c},K)$ as a \emph{conformal optical geometry}.
	\end{defn}
	It is clear that such a structure enjoys the same basic properties of its Lorentzian counterpart such as the filtration \eqref{eq-K-filt} and exact sequences \eqref{eq-seseq-K}, \eqref{eq-seseq-AnnK} and \eqref{eq-seseq-TM}. The only difference is that now the conformal structure $\mbf{c}$ induces a bundle conformal structure $\mbf{c}_K$ on the screen bundle $H_K$. In particular, $\bm{g}_{a b}$ yields a conformal bundle metric $\bm{h}_{i j} \in \Gamma( \odot^2 H^*_K [2])$.
	
	Now, let $k \in \Gamma(K)$ and set $\bm{\kappa} = \bm{g}(k,\cdot)$. Under a change of metrics in $\mbf{c}$, we have
	\begin{align*}
		\wh{\nabla}_a \bm{\kappa}_b & = \nabla_a \bm{\kappa}_b + 2 \Upsilon_{[a} \bm{\kappa}_{b]}+  \bm{g}_{a b} \Upsilon^c \bm{\kappa}_c \, ,
	\end{align*}
	where $\nabla_a$ and $\wh{\nabla}_a$ denote the respective Levi-Civita connections of $g$ and $\wh{g}$.
	Further, the Lie derivative of $\bm{\kappa}_a$ is found to be
	\begin{align}\label{eq-Lie-der-kappa}
		\mathsterling_k \bm{\kappa}_a & =  k^b \nabla_b \bm{\kappa}_a - \frac{2}{n+2} \bm{\kappa}_a \nabla_b k^b 
	\end{align}
	Note that this depends on the choice of generator $k$, since for any other generator $\wt{k} = \e^\varphi k$, the weighted $1$-form $\wt{\bm{\kappa}} = \bm{g}(\wt{k},\cdot)$ satisfies
	\begin{align*}
		\mathsterling_{\wt{k}} \wt{\bm{\kappa}} & =  \e^{2 \varphi} \left( \mathsterling_k \bm{\kappa} + \frac{n}{n+2} \bm{\kappa} \right) \, .
	\end{align*}

	\begin{lem}
		We have, for any $k \in \Gamma(K)$,
		\begin{align*}
			\mathsterling_k (k\hook \bm{\varepsilon}) & = 0 \, .
		\end{align*}
	\end{lem}
	
	\begin{proof}This follows from the formula
		\begin{multline*}
			\mathsterling_k k^b \bm{\varepsilon}_{b a_1 \ldots a_{n+1}} =  (k^b \nabla_b k^c )\bm{\varepsilon}_{c a_1 \ldots a_{n+1}} \\ + (-1)^n (n+1)k^c \bm{\varepsilon}_{c b [a_1 \ldots a_{n}} \nabla_{a_{n+1}]} k^b 
			- k^c \bm{\varepsilon}_{c a_1 \ldots a_{n+1}} \nabla_b k^b \, .
		\end{multline*}
		We know that this expression must be proportional to $k^b \bm{\varepsilon}_{b a_1 \ldots a_{n+1}}$. Contracting with $\bm{\varepsilon}^{d a_1 \ldots a_{n+1}}$ yields the factor, which turns out to be equal to zero.
	\end{proof}

	We will also examine the geometric properties of a conformal optical structure in relation to the leaf space of its associated congruence $\mc{K}$ of null curves.
	
	Any pair of metrics $g$ and $\wh{g}$ in $\mbf{c}$ single out optical geometries $(\mc{M},g,K)$ and $(\mc{M},\wh{g},K)$, and the pertinent question is how the intrinsic torsions $\mathring{T}$ and $\wh{\mathring{T}}$ of the respective geometries relate to each other. Computing the optical invariants, we find
	\begin{align*}
		\wh{\breve{\gamma}}_{i} & = \e^{2\varphi} \breve{\gamma}_i   \, ,
	\end{align*}
where $\breve{\gamma}_i$ and $\wh{\breve{\gamma}}_{i}$ are the sections of $H_K^*(2)$ with respect to $g$ and $\wh{g}$ that are obstructions to $\mc{K}$ being geodesic. In particular, we have a well-defined section $\breve{\bm{\gamma}}_i$ of $H_K^*(2)[2]$ that is the obstruction to the null curves of $\mc{K}$ being geodesics, a property well-known to be conformally invariant.

	With reference to \eqref{eq-Lie-der-kappa}, we obtain:
	\begin{lem}
		Let $(\mc{M}, \mbf{c} ,K)$ be a conformal optical geometry with congruence of null curves $\mc{K}$. Then the curves are $\mc{K}$ are geodesics if and only if 
		\begin{align*}
			\bm{\kappa} \wedge \mathsterling_k \bm{\kappa}  & = 0 \, .
		\end{align*}
	\end{lem}
	We note that the condition that the geodesics of $\mc{K}$ be affinely parametrised with respect to a generator $k$ is \emph{not} conformally invariant. We shall see however in Corollary \ref{cor-special} below that there exists a family of preferred generators of $\mc{K}$ for which $\mathsterling_k \bm{\kappa} = 0$.

	\subsection{Optical invariants}
	
	Now, if $(\mc{M},g,K)$ (and thus $(\mc{M},\wh{g},K)$) admits a congruence of null geodesics tangent to $K$, i.e., $\breve{\gamma}_i = 0$ (and thus $\wh{\breve{\gamma}}_i = 0$), we can compute the optical invariants of the respective optical geometries. We find that the twist and the shear of $\mc{K}$ transform conformally as
	\begin{align}\label{eq-conf-inv-tw-sh}
		\wh{\breve{\tau}}_{i j} & = \e^{2\varphi} \breve{\tau}_{i j}  \, , &
		\wh{\breve{\sigma}}_{i j} & = \e^{2\varphi} \breve{\sigma}_{i j} \, .
	\end{align}
	We can thus extend the definition of the shear and twist of a congruence of null geodesics in the context of a conformal optical geometry.
	\begin{defn}
		Let  $(\mc{M},\mbf{c},K)$ be a conformal optical geometry with congruence of null geodesics $\mc{K}$. Let $k$ be a generator of $\mc{K}$ and set $\bm{\kappa} = g(k,\cdot)$. Then the \emph{twist} and the \emph{shear} of $k$ are the respective sections $\bm{\tau} \in \Gamma(\bigwedge^2 H^*_K[2])$ and $\bm{\sigma} \in \Gamma( \odot^2_\circ H^*_K[2])$ defined by
		\begin{subequations}\label{eq-conf-TS}
			\begin{align}
				\bm{\tau} (v + K , w + K) & = \d^\nabla \bm{\kappa} (v,w) \, , & v,w & \in \Gamma(K^\perp) \, , \label{eq-conf-twist} \\
				\bm{\sigma} (v +K , w+K) \otimes \bm{\kappa} & = \frac{1}{2} \left( \mathsterling_{k} \bm{g} (v , w ) \otimes \bm{\kappa} - \bm{g} (v , w ) \otimes \mathsterling_{k} \bm{\kappa} \right) \, , & v, w & \in \Gamma(K^\perp) \, .\label{eq-conf-shear} 
			\end{align}
		\end{subequations}
The \emph{twist} and the \emph{shear of $\mc{K}$} are the respective sections $\breve{\bm{\tau}} \in \Gamma(\bigwedge^2 H^*_K (1)[2])$ and $\breve{\bm{\sigma}} \in \Gamma( \odot^2_\circ H^*_K (1)[2])$, whose trivialisations by some generator $k$ of $\mc{K}$ are given by $\bm{\tau}$ and $\bm{\sigma}$ above.
	\end{defn}
	As a consequence, we obtain the following conformal invariants of optical geometries.
	\begin{prop}
		Let $(\mc{M}, \mbf{c},K)$ be a conformal optical geometry. Let $g$ and $\wh{g}$ be two metrics in $\mbf{c}$ so that $(\mc{M}, g , K)$ and $(\mc{M}, \wh{g}, K)$ are optical geometries. Let $\mathring{T}$ and $\wh{\mathring{T}}$ be the intrinsic torsions of the respective optical geometries. Then
		\begin{align*}
			\mathring{T} & \in \Gamma(\mc{G}^{-1}) &  \Longleftrightarrow & & \wh{\mathring{T}} & \in \Gamma(\mc{\wh{G}}^{-1}) \, , \\
			\mathring{T} & \in \Gamma(\slashed{\mc{G}}_{-1}^1) &  \Longleftrightarrow & & \wh{\mathring{T}} & \in \Gamma(\wh{\slashed{\mc{G}}}_{-1}^1) \, , \\
			\mathring{T} & \in \Gamma(\slashed{\mc{G}}_{-1}^2) &  \Longleftrightarrow & & \wh{\mathring{T}} & \in \Gamma(\wh{\slashed{\mc{G}}}_{-1}^2) \, .
		\end{align*}
	\end{prop}
	
	\subsection{Non-expanding subclass of metrics}
	On the other hand, the expansion of some generator $k$ transforms in a non-conformally invariant way as
	\begin{align}\label{eq-trans-exp}
		\wh{\epsilon} & = \epsilon + n \Upsilon_{c} k^{c} \, .
	\end{align}
	This means that $\mc{K}$ has the same expansion with respect with both metrics $g$ and $\wh{g}$ provided these are conformally related by a factor constant along $K$. In addition, one can always use equation \eqref{eq-trans-exp} to find a metric $\wh{g}$ in $\mbf{c}$ for which the congruence is non-expanding. Indeed, the equation $\epsilon + n \Upsilon_{c} k^{c}=0$ is a first order ordinary differential equation, which always has solutions.
	\begin{prop}\label{prop-non-exp-conf}
		Let $(\mc{M}, \mbf{c} ,K)$ be a conformal optical geometry with congruence of null geodesics $\mc{K}$. Then locally, there is a subclass $\accentset{n.e.}{\mbf{c}}$ of metrics in $\mbf{c}$ with the property that whenever $g$ is in $\accentset{n.e.}{\mbf{c}}$, the congruence $\mc{K}$ is non-expanding, i.e., for any $k \in \Gamma(K)$ with $\kappa = g(k,\cdot)$, $\kappa \, \div  k -  \nabla_k \kappa = 0$.
		
		Any two metrics in $\accentset{n.e.}{\mbf{c}}$ differ by a factor constant along $K$.
		
		The conformal subclass $\accentset{n.e.}{\mbf{c}}$ induces a conformal subclass $\accentset{n.e.}{\mbf{c}}_K$ of $\mbf{c}_K$ on the screen bundle.
	\end{prop}
	
	In other words, any metric $g$ in the subclass $\accentset{n.e.}{\mbf{c}}$ defines an optical geometry $(\mc{M}, g, K)$ whose intrinsic torsion is a section of $\slashed{\mc{G}}_{-1}^0$.
	
	As corollary of Proposition \ref{prop-non-exp-conf}, we have:
	\begin{cor}\label{cor-special}
		Let $(\mc{M}, \mbf{c} ,K)$ be a conformal optical geometry with congruence of null geodesics $\mc{K}$.
		Then, locally, there exists a family of local generators $k \in \Gamma(K)$ such that
		\begin{align}\label{eq-special-gen}
			\mathsterling_k \bm{\kappa} & = 0 \, .
		\end{align}
		These generators have the property that they generate affinely parametrised geodesics tangent to $K$ for any choice of metric in $\accentset{n.e.}{\mbf{c}}$.
	\end{cor}
	
	\begin{defn}
		Let $(\mc{M},\mbf{c},K)$ be a conformal optical geometry. We say that a generator $k$ of $K$ is \emph{special} if 
		$\bm{\kappa}_a = k^b \bm{g}_{a b}$ satisfies \eqref{eq-special-gen}.
	\end{defn}

	\subsection{Class of metrics conformal to Walker metrics}
	Finally, assuming that $\mc{K}$ is non-expanding, non-twisting and non-shearing, we compute
	\begin{align}\label{eq-trans-par}
		\wh{E}_{i} & = \e^{2\varphi} \left( E_{i}  - \Upsilon_{i} \right) \, .
	\end{align}
	This means that $\mc{K}$ has the same obstruction to parallelism with respect with both metrics $g$ and $\wh{g}$ provided these are conformally related by a factor constant along $K^\perp$. In addition, using equations \eqref{eq-trans-exp} and \eqref{eq-trans-par}, we prove the following result.
	\begin{prop}\label{prop:conf-to-Walker}
		Let $(\mc{M}, \mbf{c} ,K)$ be a conformal optical geometry with non-twisting non-shearing congruence of null geodesics $\mc{K}$. Let $k \in \Gamma(K)$ and set $\bm{\kappa}_a = \bm{g}_{a b}k^b$. Suppose the Weyl tensor $W_{a b c d}$ satisfies
		\begin{align}\label{eq:cond-Walker}
			k^a W_{a b [c d} \bm{\kappa}_{e]} & = 0 \, .
		\end{align}
		Then locally, there is a subclass $\accentset{par.}{\mbf{c}}$ of metrics in $\mbf{c}$ with the property that whenever $g$ is in $\accentset{par.}{\mbf{c}}$, the line distribution $K$ is parallel, i.e., for any $k \in \Gamma(K)$ with $\bm{\kappa} = \bm{g}(k,\cdot)$, $\nabla_v \bm{\kappa} \wedge \bm{\kappa} = 0$.
		
		Any two metrics in $\accentset{par.}{\mbf{c}}$ differ by a factor constant along $K^\perp$.
	\end{prop}
	
	\begin{proof}
		With no loss, we can restrict ourselves to a metric $g$ in the subclass $\accentset{n.e.}{\mbf{c}}$, so that since $\mc{K}$ is non-twisting and non-shearing, the metric $g$ is a Kundt metric. We already know that the condition \eqref{eq:cond-Walker} is a necessary condition for the optical structure to be integrable. We now show that this is also sufficient locally. In particular, the condition \eqref{eq:cond-Walker} is equivalent to
		\begin{align*}
			k^a \delta^b_i k^c \delta^d_j W_{a b c d} = \left( k^a \delta^b_i \delta^c_j \delta^d_k W_{a b c d} \right)_\circ & = 0 \, , \\
			k^a \ell^b k^c \delta^d_j W_{a b c d} & = 0 \, , \\
			k^a \ell^b \delta^c_i \delta^d_j W_{a b c d} & = 0 \, .
		\end{align*}
		The first of these conditions is always satisfied for a non-twisting non-shearing congruence of null geodesics. The second one is equivalent to $\mathsterling_k E_i = 0$, and the third one is equivalent to $\ul{\nabla}_{[i} E_{j]}=0$, i.e., locally $E_i = (\d f)_i$ for some smooth function $f$ on the leave space of $\mc{K}$. Now, using \eqref{eq-trans-par}, we can find a metric for which $K$ is parallel.
	\end{proof}
	
	For related results on the conformal geometry, see \cite{leistner05a,Leistner2012}.

\subsection{Non-twisting congruences}\label{sec-non-twist-conf-optical}
This is a direct consequence of Proposition \ref{prop-non-twisting}:
\begin{prop}\label{prop-non-twist-conf-optical}
Let $(\mc{M}, \mbf{c},K)$ be a conformal optical geometry with congruence of null geodesics $\mc{K}$ with leaf space $(\ul{\mc{M}},\ul{H})$.  The following statements are equivalent:
\begin{enumerate}
	\item $\mc{K}$ is non-twisting;
	\item $\mc{M}$ is locally foliated by a one-parameter family of $(n+1)$-dimensional submanifolds, each containing an $n$-parameter family of null geodesics of $\mc{K}$;
	\item $\ul{H}$ is integrable, i.e., $\ul{\mc{M}}$ is locally foliated by $n$-dimensional submanifolds tangent to $\ul{H}$.
\end{enumerate}
\end{prop}

\subsection{Non-shearing congruence}\label{sec-non-shear-conf-optical}
This is a direct consequence of Propositions \ref{prop-non-shearing} and \ref{prop-non-exp-conf}:
\begin{prop}\label{proposition-shfr-Lie-conf}
Let $(\mc{M}, \mbf{c},K)$ be a conformal optical geometry with congruence of null geodesics $\mc{K}$. The following statements are equivalent.
\begin{enumerate}
	\item $\mc{K}$ is non-shearing.
	\item The induced conformal structure $\mbf{c}_K$ on the screen bundle $H_K$ is preserved along the geodesics of $\mc{K}$.
	\item The induced conformal structure $\mbf{c}_K$ on $H_K$ descends to a conformal structure $\ul{\mbf{c}}$ on $\ul{H}$. More precisely, there is a one-to-one correspondence between metrics in $\accentset{n.e.}{\mbf{c}}_K$ and metrics in $\ul{\mbf{c}}$.
\end{enumerate}
\end{prop}

\subsection{Non-twisting non-shearing spacetimes}\label{sec-non-twist-non-shear-conf-optical}
Now, combining Propositions \ref{prop-non-twist-conf-optical} and \ref{proposition-shfr-Lie-conf} yields:
\begin{prop}\label{proposition-SF->intconfH-conf}
Let $(\mc{M},\mbf{c},K)$ be a conformal optical geometry with congruence of null geodesics $\mc{K}$ with leaf space $(\ul{\mc{M}},\ul{H})$. Then the following statements are equivalent.
\begin{enumerate}
	\item $\mc{K}$ is non-twisting non-shearing.
	\item $\mc{K}$ is foliated by $(n+1)$-dimensional submanifolds tangent to $H$, each of which inherits a conformal structure from $\mbf{c}_K$.
	\item $\ul{\mc{M}}$ is foliated by $n$-dimensional submanifolds tangent to $\ul{H}$, each of which inherits a conformal structure $\ul{\mbf{c}}$ from $\mbf{c}_K$. Further, there is a one-to-one correspondence between  metrics in $\ul{\mbf{c}}$ and metrics in $\accentset{n.e.}{\mbf{c}}_K$.
\end{enumerate}
\end{prop}
In other words, a Robinson-Trautman spacetime is locally conformal to a Kundt spacetime.

\subsection{Twisting congruence of null geodesics}
We can extend the definition of the rank of the twist to the conformal setting:
\begin{defn}
The \emph{rank} of the twist of the congruence $\mc{K}$ is the rank of any section $\kappa \in \Gamma(\Ann(K^\perp))$.
\end{defn}

\subsubsection{Maximally twisting congruences}
From Proposition \ref{proposition-maxTw-contact}, we immediately obtain
\begin{prop}\label{proposition-maxTw-contact_conf}
Let $(\mc{M}, \mbf{c},K)$ be a $(2m+2)$-dimensional conformal optical geometry with maximally twisting congruence of null geodesics $\mc{K}$. Then the local leaf space $(\ul{\mc{M}},\ul{H})$ of $\mc{K}$ is equipped with a contact structure.

Further, every choice of special generator of $\mc{K}$ establishes a one-to-one correspondence between metrics in $\accentset{n.e.}{\mbf{c}}$ and contact forms on $\ul{\mc{M}}$.
\end{prop}

In odd dimensions, we have:
\begin{prop}
Let $(\mc{M},\mbf{c},K)$ be a $(2m+3)$-dimensional conformal optical geometry with maximally twisting congruence of null geodesics  $\mc{K}$ and leaf space $(\ul{\mc{M}}, \ul{H})$. Then,  for any non-vanishing section $\ul{\kappa}$ of $\Ann(\ul{H})$, the distribution $\ul{H}' :=\ker \d \underline{\kappa} \cap \ker \underline{\kappa}$ is a line subbundle of $\ul{H}$, and $\ul{H}/\ul{H}'$ descends to a contact distribution on the $(2m+1)$-dimensional leaf space of $\ul{H}'$.
\end{prop}

\subsubsection{Maximally twisting non-shearing congruences}

If $(\ul{\mc{M}},\ul{H})$ is a contact manifold, a \emph{sub-conformal contact structure} $\ul{\mbf{c}}$ on $\ul{\mc{M}}$ is a conformal structure  on $\ul{H}$. We refer to $(\ul{\mc{M}},\ul{H},\ul{\mbf{c}})$ as a \emph{sub-conformal contact manifold}. Any choice of metric in $\ul{\mbf{c}}$ defines a \emph{sub-Riemannian structure} on $\ul{H}$, and each choice of contact $1$-form in determines a metric in $\ul{\mbf{c}}$ -- see for instance \cite{Falbel2007,Falbel2014,Eastwood2016}.

The following result gives a very neat relation between choices of metrics in the conformal class and splittings of the optical structure.
\begin{prop}\label{prop-adapted_frame_max_tw-nsh-nexp_conf}
Let $(\mc{M}, \mbf{c}, K)$ be a  $(2m+2)$-dimensional conformal optical geometry with maximally twisting non-shearing congruence of null geodesics $\mc{K}$. Then for each $g$ in $\accentset{n.e.}{\mbf{c}}$, there exists a unique pair $(k, \ell)$,  where $k$ is a generator of $\mc{K}$ and $\ell$ a null vector field such that $g(k, \ell) = 1$, and  $\kappa = g( k , \cdot)$ satisfies
\begin{align*}
	\d \kappa (k , \cdot) & = 0 \, , & \d \kappa (\ell , \cdot) & = 0 \, ,
\end{align*}
and the twist of $k$ satisfies $\tau_{i j} \tau^{i j} = 2m$.

Further, the leaf space $(\ul{\mc{M}}, \ul{H})$ of $\mc{K}$ acquires a subconformal contact structure $\ul{\mbf{c}}$ whereby each metric $g$ in $\accentset{n.e.}{\mbf{c}}$ descends to a bundle metric in $\ul{\mbf{c}}$, and $\kappa = g(k, \cdot)$ descends to the corresponding contact $1$-form.

Let $(k, \ell)$ and $(\wh{k}, \wh{\ell})$ be any two such pairs corresponding to metrics $g$ and $\wh{g}$ in $\accentset{n.e.}{\mbf{c}}$, with $\wh{g} = \e^{2 \varphi} g$ for some smooth function $\varphi$ constant along $K$. Then, $\wh{k}^a = k^a$, and $\kappa = g(k,\cdot)$, $\wh{\kappa} = \wh{g}(\wh{k},\cdot)$, $\lambda = g(\ell,\cdot)$, $\wh{\lambda} = \wh{g}(\wh{\ell},\cdot)$, $\widehat{\delta}_a^i$ and $\delta_a^i$ are related via
\begin{align*}
	\widehat{\kappa}_a & =  \e^{2 \varphi} \kappa_a \, , \\
	\widehat{\lambda}_a & =  \lambda_a + h_{i j} (\tau^{-1})^{i k} \Upsilon_k \delta_a^j - \frac{1}{2} h_{i j} (\tau^{-1})^{i k} (\tau^{-1})^{j \ell} \Upsilon_k \Upsilon_\ell \kappa_a  \, , \\
	\widehat{\delta}_a^i & = \delta_a^i - (\tau^{-1})^{i j} \Upsilon_j \kappa_a \, ,
\end{align*}
where $\Upsilon_i := \delta^a_i \nabla_a \varphi$.
\end{prop}

\begin{proof}
The first part of the proposition is a direct consequence of Proposition \ref{prop-adapted_frame_max_tw-nsh-nexp}.

The geometric interpretation of the leaf space as a contact subconformal manifold follows directly from Proposition \ref{proposition-maxTw-contact}.

For the remainder, let $(k, \ell)$ and $(\wh{k}, \wh{\ell})$ be as given in the proposition. The relation between their respective twists are given by $\wh{\tau}_{i j} = \e^{2 \varphi} \tau_{i j}$. But we know $(\d^\nabla \kappa)_{a b} = \tau_{i j} \delta^i_a \delta^j_b$ and $(\d^{\wh{\nabla}} \wh{\kappa})_{a b} = \wh{\tau}_{i j} \wh{\delta}^i_a \wh{\delta}^j_b$. Using the relation between $\nabla$ and $\wh{\nabla}$, we find
\begin{align*}
	\wh{\tau}_{i j} \wh{\delta}^i_a \wh{\delta}^j_b = \e^{2 \varphi} \tau_{i j} \delta^i_a \delta^j_b + 2 \e^{2 \varphi} \Upsilon_i \delta^i_{[a} \kappa_{b]} \, ,
\end{align*}
from which we deduce the required change for $\delta^i_a$. The required change for $\lambda$ can be deduced by comparing the change of metrics, or by simply noting that $\wh{\ell}^a = \wh{g}^{a b} \wh{\lambda}_b$ should be null and annihilate $\wh{\delta}^i_a$.
\end{proof}
The above proposition is applied in \cite{TaghaviChabert2022} to great effect in determining Einstein metrics in the conformal class: such a metric is essentially determined by a $1$-form $\lambda$ as in the proposition, and a conformal factor that depends only on an affine parameter along the geodesics of $\mc{K}$.

Just as for metric optical structures, one can define the twist endomorphism of $H_K = K^\perp/K$ as in \eqref{eq-twist-endo} in the conformal setting as the section of $\mathrm{End}(H_K)(1)$ given by
\begin{align}\label{eq-twist-endo-conf}
F & := \bm{h}^{-1} \circ \bm{\tau} \, .
\end{align}
Note that $F$ does not depend on the choice of metric in $\mbf{c}$.

Mirroring Proposition \ref{prop-twist-endo}, we have:
\begin{prop}\label{prop-twist-endo-conf}
Let $(\mc{M},\mbf{c},K)$ be a conformal optical geometry with non-shearing congruence of null geodesics $\mc{K}$ with leaf space $(\ul{\mc{M}}, \ul{H})$. Let $k$ be any special generator $k$ of $\mc{K}$. Then the twist endomorphism $F$ associated to $k$ descends to an endomorphism of $\ul{H}$ on $\ul{\mc{M}}$.
\end{prop}

\begin{exa}[Fefferman's construction]
When the endomorphism $\ul{F}$ associated to the contact sub-conformal structure is a complex structure $\ul{J}$, then we refer to  $(\ul{\mc{M}},\ul{H},\ul{\mbf{c}})$ as a \emph{partially integrable almost CR manifold}. In the case when the eigenbundles of $\ul{J}$ are integrable, Fefferman showed that one can construct a conformal structure of Lorentzian signature on a circle bundle $\mc{M} \accentset{\varpi}{\rightarrow} \ul{\mc{M}}$ in a canonical way: for each contact form $\ul{\theta}$ and corresponding bundle metric $\ul{h}$ in $\ul{\mbf{c}}$, the metric on $\mc{M}$ is given by
\begin{align*}
	g & = 4 \varpi^* \ul{\theta}\, \lambda + \varpi^* \ul{h} \, .
\end{align*}
where $\lambda$ is a certain canonically defined $1$-form on $\mc{M}$ that does not vanish on restriction to the fibers of $\mc{M} \rightarrow \ul{\mc{M}}$.

This construction was generalised to the non-integrable case in \cite{Leitner2010} and to the case when $\ul{F}$ is not a complex structure, but has constant eigenvalues in \cite{Falbel2007,Falbel2014}. See also the recent work \cite{Alekseevsky2021} and \cite{Taghavi-Chabert2012}.
\end{exa}

\section{Four-dimensional case}\label{sec:four_dim}
The four-dimensional case is very special. We note that if $(\mc{M},\mbf{c},K)$ is a four-dimensional conformal optical geometry, there is a volume form $\bm{\varepsilon}$ of conformal weight $4$, which induces a skew-symmetric bilinear form $\bm{\varepsilon}_K$ of weight $2$ on the screen bundle $H_K = K^\perp/K$. This additional structure is particularly useful: it allows us to construct a bundle complex structure on $H_K$. The following proposition is immediate.
\begin{prop}\label{prop-ANSt-4}
Let $(\mc{M},\mbf{c},K)$ be a four-dimensional conformal optical geometry. Let $\bm{\varepsilon}_K$ the weighted volume form on the screen bundle $H_K = K^\perp/K$. Then the bundle endomorphism on $H_K$ defined
\begin{align}\label{eq-canon-J}
	{J} & := \bm{h}^{-1} \circ \bm{\varepsilon}_K \, ,
\end{align}
is a bundle complex structure on $H_K$ compatible with the conformal structure $\mbf{c}_K$ on $H_K$, and with eigenbundles $N/{}^\C K$ and $\overline{N}/{}^\C K$ where $N$ is a  totally null complex rank-$2$ distribution, $\overline{N}$ its complex conjugate, and $N \cap \overline{N} = {}^\C K$.

Conversely, a totally null complex rank-$2$ distribution $N$ defines an optical structure $K$ on $(\mc{M},\mbf{c})$, and thus a bundle complex structure on $H_K$ compatible with the conformal structure.
\end{prop}

\begin{rem}\label{rob-remark}
The pair $(N,K)$, or equivalently $(K,J)$, defines an \emph{almost Robinson structure} on $(\mc{M}, \mbf{c})$, or, with a choice of metric $g$, on $(\mc{M}, g)$, as first defined in \cite{Nurowski2002,Trautman2002,Trautman2002a}, and is the subject of \cite{Fino2023}. The above proposition tells us that in dimension four there is no distinction between (conformal) optical structures and almost Robinson structures.
\end{rem}

When it comes to the integrability of the almost Robinson structure, we have the following well-known theorem -- see  e.g.\  \cite{Robinson1985,Penrose1986} and the aforementioned references.
\begin{thm}\label{thm-shfr-ANSt-4}
Let $(\mc{M},\mbf{c},K)$ be an oriented four-dimensional conformal optical geometry with congruence of null curves $\mc{K}$. Let $J$ be the screen bundle endomorphism \eqref{eq-canon-J} and $N$ the associated complex totally null $2$-plane distribution defined in Proposition \ref{prop-ANSt-4}. Then the following three statements are equivalent:
\begin{enumerate}
	\item $N$ is involutive, i.e., closed under the bracket of vector fields.\label{item:N-inv}
	\item $J$ is preserved along the flow of any generator of $\mc{K}$.\label{item:J_pres}
	\item $\mc{K}$ is a non-shearing congruence of null geodesics.\label{item:non-sh}
\end{enumerate}
Further, if any of these conditions is satisfied, then the leaf space $(\ul{\mc{M}},\ul{H},\mbf{c}_{\ul{H}})$ is equipped with a CR structure, that is, $\ul{H}$ is equipped with a bundle complex structure.
\end{thm}

\begin{proof}
Since $N$ is totally null and of rank $2$, and ${}^\C K \subset N$, the condition that $N$ be involutive is that $g([k,v],w) =0$ for any sections $v$, $w$ of $N$ and $k$ of $K$, and metric $g$ in $\mbf{c}$. With no loss we can choose $v$ and $k$ such that $v \wedge k \neq 0$. But this is equivalent to $g(\mathsterling_k v ,w) =0$, i.e., the eigenbundle $N/{}^\C K$ of $J$ is preserved along the flow of $k$, and similarly for its complex conjugate. This proves the equivalence of \eqref{item:N-inv} and \eqref{item:J_pres}.

Now, for any section $k$ of $K$, $v, w$ sections of $N$ so that $g(v,w) = 0$ for any metric $g$ in $\mbf{c}$, we have, using the Leibniz rule,
\begin{align*}
	0 & = \mathsterling_k (g(v,w)) = (\mathsterling_k g) (v,w) + g([k, v],w) + g([k, w],v) \, . 
\end{align*}
Since $N$ has rank two, and ${}^\C K \subset N$, we have three possibilities:
\begin{align*}
	0 & =  (\mathsterling_k g) (k,k) \, , \\
	0 & =  (\mathsterling_k g) (v,k) + g([k, v],k) \, , \\
	0 & =  (\mathsterling_k g) (v,v) + 2 g([k, v],v) \, . 
\end{align*}
The first condition is vacuous, while the other two together with their complex conjugate tells us that $\mc{K}$ is non-shearing if and only if $N$ is involutive (together with $\overline{N}$.)

Thus, all conditions \eqref{item:N-inv}, \eqref{item:J_pres} and \eqref{item:non-sh} are equivalent.

For the last part, we note that since $J$ is preserved along the flow of $k$ together with its eigenbundles, it descends to a complex structure on $\ul{H}$ whose eigenbundles are necessarily involutive since they have rank one. Hence the leaf space $(\ul{\mc{M}},\ul{H},\mbf{c}_{\ul{H}})$ is equipped with a CR structure.
\end{proof}

\begin{rem}
	When the optical structure is tangent to a \emph{twisting} non-shearing congruence of null geodesics, then the underlying CR three-manifold is non-degenerate (i.e.\ contact). It was recently established that if, in addition, certain conditions are imposed on the Weyl tensor and Bach tensor, then $(\mc{M},\mbf{c},K)$ is locally conformally isometric to a `perturbation' of Fefferman's canonical conformal bundle by a semi-basic $1$-form --- see \cite{TaghaviChabert2023} for details.
\end{rem}

\section{Generalised optical geometries}\label{sec:gen_opt}

\subsection{Generalised optical structures}
Generalised optical geometries were first introduced by Robinson and Trautman in \cite{Trautman1984,Trautman1985,Robinson1985,Robinson1986,Robinson1989,Trautman1999}, where they are referred to simply as `optical geometries'. Two equivalent definitions are given there, and we shall give the one that generalises the notion of conformal optical geometries first. The alternative definition shall be considered later.
\begin{defn}\label{def:gen_opt_str}
Let $\mc{M}$ be a smooth manifold. A \emph{generalised optical structure} consists of a pair $(K,\mbf{o})$,  where $K$ is a line distribution on $\mc{M}$, and $\mbf{o}$ an equivalence class of Lorentzian metrics such that
\begin{enumerate}
	\item for each $g$ in $\mbf{o}$, $K$ is null with respect to $g$;
	\item any two metrics $g$ and $\wt{g}$ in $\mbf{o}$ are related by
	\begin{align}\label{eq:optical-metric}
		\wt{g} & = \e^{2 \varphi} \left( g + 2 \, \kappa \, \alpha \right) \, ,
	\end{align}
	for some smooth function $\varphi$ and $1$-form $\alpha$ on $\mc{M}$, and $\kappa = g(k ,\cdot)$ for any non-vanishing section $k$ of $K$.
\end{enumerate}
We refer to $(\mc{M},K,\mbf{o})$ as a \emph{generalised optical geometry}.
\end{defn}

\begin{rem}
\begin{enumerate}
	\item The two conditions above are well-defined: if $K$ is null with respect to a given metric $g$, then it is null with respect to any metric $\wt{g}$ related to $g$ via \eqref{eq:optical-metric}.
	\item Similarly, if $K^\perp$ is the orthogonal complement to $K$ with respect to one metric $g$ in $\mbf{o}$, then $K^\perp$ remains the orthogonal complement to $K$ for any other metric $g$ in $\mbf{o}$. So the notion of orthogonal complement $K^\perp$ with respect to $\mbf{o}$ is well-defined.
	\item The requirement that $\wt{g}$ in \eqref{eq:optical-metric} must be non-degenerate implies that $\alpha(k) \neq -1$.
	\item If a metric $g$ is in $\mbf{o}$, so is any metric in its conformal class $[g]$.
\end{enumerate}
\end{rem}

\begin{rem}
When $\alpha \wedge \kappa = 0$, the metric $\wt{g}$ given in \eqref{eq:optical-metric} reduces to 
\begin{align*}
	\wt{g} & = \e^{2 \varphi} \left( g + F \kappa^2 \right) \, ,
\end{align*}
for some smooth function $F$. We can then view $\wt{g}$ as an exact first-order perturbation of $g$ up to some overall factor. Its inverse is given by
\begin{align*}
	\wt{g}^{-1} = \e^{-2 \varphi} \left( g^{-1} - F k^2 \right) \, .
\end{align*}
If in addition $\varphi=0$, the metric $\wt{g}$ is known as a \emph{(generalised) Kerr-Schild metric}, and if $g$ is the Minkowski metric, simply as a \emph{Kerr-Schild metric}.
\end{rem}

A generalised optical structure $(\mc{M}, K, \mbf{o})$ has an associated congruence of null curves $\mc{K}$ tangent to $K$. The next proposition deals with the question of which geometric properties of the congruence are shared by all metrics in $\mbf{o}$.

\begin{prop}\label{prop:gen_tw-sh}
Let $(\mc{M}, K, \mbf{o})$ be a generalised optical geometry with congruence of curves $\mc{K}$. Then $K$ is null with respect to any metric in $\mbf{o}$. Further, if the curves of $\mc{K}$ are geodesics with respect to one metric in $\mbf{o}$, they are geodesics for any other. This being the case, the twist and the shear of $\mc{K}$ do not depend on the choice of metric in $\mbf{o}$.

In other words, for any metrics $g$ and $\wt{g}$ in $\mbf{o}$, the two conformal optical geometries $(\mc{M}, [g], K)$ and $(\mc{M}, [\wt{g}], K)$ share the same optical properties.
\end{prop}

\begin{proof}
Let $k$ be a generator of $\mc{K}$, $g$ and $\wt{g}$ two metrics in $\mbf{o}$. Then $\wt{\kappa} = \wt{g}(k,\cdot) = \e^{2\varphi} (1 + \alpha(k)) \kappa$. The result the follows from Lemma \ref{lem-proposition-geod} and the defining equations \eqref{eq-twist} and \eqref{eq-shear} for the twist and the shear, together with the fact that the Lie derivative and the exterior derivative do not depend on the metric.
\end{proof}

We can therefore talk of a congruence of null geodesics $\mc{K}$ for $(\mc{M}, K, \mbf{o})$ with given shear and twist. We also note that Robinson and Trautman gave an interesting characterisation of congruences of null geodesics in the context of generalised optical geometries in \cite{Robinson1986a} along the lines of the Robinson--Mariot theorem \cite{Mariot1954,Robinson1961}.

Proposition \ref{prop:gen_tw-sh} tells that two distinct (conformal) optical geometries may share the same congruence of null geodesics with the same twist and shear. This can be made more explicit by starting with an $(n+1)$-dimensional manifold $\ul{\mc{M}}$, which we extend to the trivial line bundle $\mc{M} = \R \times \ul{\mc{M}}$.
We shall construct a Lorentzian metric on $\mc{M}$ and identify $\ul{\mc{M}}$ as the leaf space of a congruence of null geodesics. Denote by $\varpi$ the natural projection from $\mc{M}$ to $\ul{\mc{M}}$, and the line distribution $K$ on $\mc{M}$ defined by $\ker \varpi_*$ is tangent to a congruence $\mc{K}$ of curves, which are none other than the fibers of the bundle $\mc{M}$.

We also assume that $\ul{\mc{M}}$ is endowed with a hyperplane distribution $\ul{H}$. Let $\ul{\kappa}$ be a $1$-form annihilating $\ul{H}$, and extend $\ul{\kappa}$ to a coframe $(\ul{\kappa}, \ul{\theta}^i)$ for $\ul{\mc{M}}$. Set $\kappa = \varpi^* \ul{\kappa}$ and $\theta^i = \varpi_* \ul{\theta}^i$. Now, choose a $1$-form  $\lambda$ such that $\lambda (k)$ does not vanish for any non-vanishing section $k$ of $K$, and a positive definite symmetric matrix $h_{i j}$ depending smoothly on $\mc{M}$. Then
\begin{align*}
g & = 2 \, \kappa \, \lambda + h_{i j} \theta^i \, \theta^j \, ,
\end{align*}
is a Lorentzian metric on $\mc{M}$ for which $K$ is null, and since for any section $k$ of $K$, $\mathsterling_k \kappa = 0$, the curves of $\mc{K}$ are geodesics. Hence $(\mc{M}, g, K)$ is an optical geometry.

The freedom in choosing $\lambda$, $h_{i j}$, together with the freedom in the choice of $\ul{\theta}^i$ leads to the equivalent class $\mbf{o}$ of metrics related via \eqref{eq:optical-metric}. Hence $\mc{M}$ is endowed with a generalised optical structure $(K,\mbf{o})$.

If the distribution $\ul{H}$ is involutive, i.e., $\ul{\kappa} \wedge \d \ul{\kappa} =0$, then, by the naturality of the exterior derivative, $\kappa \wedge \d \kappa =0$, i.e., $\mc{K}$ is non-twisting.

If the distribution $\ul{H}$ is equipped with a Riemannian structure $\ul{h}_{i j}$, then we can take $h_{i j} = \e^{2 \varphi} \ul{h}_{i j}$ for some function $\varphi$ on $\mc{M}$. In this case, $\mc{K}$ is a non-shearing, and if $\varphi$ is a function on $\ul{\mc{M}}$, $\mc{K}$ is also non-expanding.

\subsection{Generalised optical geometries as $G$-structures}
The original definition of a generalised optical structure of \cite{Trautman1984,Trautman1985,Robinson1985,Robinson1986,Robinson1989,Trautman1999}, is given in terms of a $G$-structure on a four-dimensional smooth manifold $\mc{M}$ where the structure group $H$ is the stabiliser of a line distribution $K^{(1)}$ and a rank-$3$ distribution $K^{(3)}$ such that $K^{(1)} \subset K^{(3)}$, and a complex structure on each fiber of $K^{(3)}/K^{(1)}$. In particular, $H$ is a subgroup of $\mbf{GL}(4,\R)$, \emph{not} $\SO(1,3)$. Noting that the complex structure on $K^{(3)}/K^{(1)}$ can be replaced by a conformal structure, this admits a straightforward generalisation to higher dimensions.

\begin{prop}\label{prop:gen-opt-G-str}
Let $\mc{M}$ be a smooth oriented $(n+2)$-dimensional manifold. Then the following statements are equivalent.
\begin{enumerate}
	\item $\mc{M}$ is endowed with a generalised optical structure $(K,\mbf{o})$.\label{item:gen_opt}
	\item $\mc{M}$ is endowed with a pair of distributions $K^{(1)}$ and $K^{(n+1)}$ of rank $1$ and $n+1$ respectively such that
	\begin{align}\label{eq:filt}
		K^{(1)} \subset K^{(n+1)}
	\end{align} and its associated screen bundle $K^{(n+1)}/K^{(1)}$ is equipped with a conformal structure of Riemannian signature.\label{item:gen_opt-G}
\end{enumerate}
\end{prop}

\begin{proof}
Assume \eqref{item:gen_opt}. With reference to Definition \ref{def:gen_opt_str}, set $K^{(1)} := K$. The notion of orthogonal complement $K^\perp$ of $K$ does not depend on the choice of metric in $\mbf{o}$, so we can set $K^{(n+1)} := K^\perp$. We have that $K^{(1)} \subset K^{(n+1)}$ and it is straightforward to check that $\mbf{o}$ induces a conformal structure of Riemannian signature on $K^{(n+1)}/K^{(1)}$: indeed, let $g$ and $\wt{g}$ be two metrics in $\mbf{o}$ related by \eqref{eq:optical-metric}, then, for any $v, w \in \Gamma(K^\perp)$, we have metrics
\begin{align*}
	h(v + K , w+ K) & = g(v,w) \, , & \wt{h}(v + K , w+ K) & = \wt{g}(v,w) \, , 
\end{align*}
on $K^{(n+1)}/K^{(1)}$ that are conformally related by $\wt{h} = \e^{2 \varphi} h$.

For the converse, assume \eqref{item:gen_opt-G}. Choose a metric $h$ in the conformal class on the screen bundle, together with vector fields $(e_i)_{i=1,\ldots,n}$ such that $(e_i + K^{(1)})_{i=1,\ldots,n}$ form an orthonormal frame on the screen bundle with respect to $h$, and set $h_{i j} = h(e_i , e_j)$. Let $(\kappa , \theta^i)_{i=1,\ldots, n}$ be a set of $1$-forms that spans $\Ann(K^{(1)})$ such that $\theta^{i}(e_j) = \delta_j^i$, and extend it to a coframe $(\kappa , \theta^i , \lambda)_{i=1,\ldots, n}$ on $\mc{M}$. Then
\begin{align*}
	g & = 2 \kappa \, \lambda + h_{i j} \theta^i \, \theta^j \, ,
\end{align*}
is a Lorentzian metric on $\mc{M}$ for which $K^{(1)}$ is null.
Now, the freedom in choosing the coframe is given by
\begin{align*}
	\wt{\kappa} & = a \kappa \, , & \wt{\theta}^i & = \phi_j{}^i \theta^j + \psi^i \kappa \, , & \wt{\lambda} & = b \lambda + c_i \theta^i + f \kappa \, ,
\end{align*}
where $a$, $\phi_j{}^i$, $\psi^i$, $b$, $c_i$ and $f$ are smooth functions on $\mc{M}$ with $a b \neq 0$, and $\phi_i{}^j$ preserves $h_{i j}$. We are also free to choose a different metric $\wt{h}_{i j} = \e^{2 \varphi} h_{i j}$ for some smooth function $\varphi$ on the screen bundle. Then the metric on $\mc{M}$ corresponding to the coframe $(\wt{\kappa} , \wt{\theta}^i , \wt{\lambda})_{i=1,\ldots, n}$ and screen bundle metric $\wt{h}_{i j}$ is given by
\begin{align*}
	\wt{g} & = 2 \wt{\kappa} \, \wt{\lambda} + \wt{h}_{i j} \wt{\theta}^i \, \wt{\theta}^j \, ,
\end{align*}
and is related to $g$ via \eqref{eq:optical-metric} where
\begin{align*}
	\alpha & = (a b \e^{-2 \varphi} - 1) \lambda + \phi_{i j} \psi^j \theta^i + \frac{1}{2} \psi_k \psi^k \kappa \, .
\end{align*}
Hence, the geometric structure \eqref{item:gen_opt-G} gives rise to a generalised optical structure.
\end{proof}

Thus, we can view a generalised optical structure on a smooth manifold $\mc{M}$ as a $G$-structure where the structure group $H$, say, of the frame bundle of $\mc{M}$ is reduced from $\mbf{SL}(n+2,\R)$ (or $\mbf{GL}(n+2, \R)$ if we drop the assumption that $\mc{M}$ is oriented) to the closed Lie subgroup $H$  that stabilises the distribution filtration \eqref{eq:filt} together with a conformal structure on its associated screen bundle. From the proof of Theorem \ref{prop:gen-opt-G-str}, we see that $H$ has dimension $\frac{1}{2}n(n+3)+4$. The generalised optical geometry $(\mc{M},K,\mbf{o})$ is integrable as a $G$-structure if and only if there exist local coordinates $(u, v, x^i)$ on $\mc{M}$ and a metric $g$ in $\mbf{o}$ and a section $k$ of $K$ such that
\begin{enumerate}
\item $g = 2 \d u \, \d v + \delta_{i j} \d x^i \, \d x^j$,  where $\delta_{i j}$ is the standard Euclidean metric on $\R^n$, i.e., $g$ is the Minkowski metric, and
\item $k = \frac{\partial}{\partial v}$.
\end{enumerate}

The following theorem generalises a result in dimension four given in \cite{Robinson1985} to higher dimensions.
\begin{thm}\label{thm:integrable_optical}
Let $(\mc{M},K,\mbf{o})$ be a generalised optical geometry with congruence of null curves $\mc{K}$. Then the following statements are equivalent:
\begin{enumerate}
	\item There exists a torsion-free linear connection compatible with $\mbf{o}$.\label{item:lin_conn}
	\item $\mc{K}$ is a non-twisting non-shearing congruence of null geodesics.\label{item:non-sh-non-sh-cong}
\end{enumerate}
Further, in dimension four ($n=2$), $(\mc{M},K,\mbf{o})$ is integrable as a $G$-structure if and only if any of the conditions \eqref{item:lin_conn} and \eqref{item:non-sh-non-sh-cong} holds.

In dimensions six and higher ($n>3$), in the neighbourhood of any point in $\mc{M}$, there exists smooth functions $u$ and $v$ such that
\begin{itemize}
	\item $\frac{\partial}{\partial v}$ spans $K$,
	\item $\d u$ annihilates $K^\perp$,
	\item $\mbf{o}$ contains the metric $g = 2 \d u \d v + \ul{h}$, where $\ul{h}$ is a family of conformally flat metrics smoothly parametrised by $u$.
\end{itemize}
if and only if any of the conditions \eqref{item:lin_conn} and \eqref{item:non-sh-non-sh-cong} holds together with the condition
\begin{align}\label{eq:W_cond}
	\left( \kappa_{[a} W_{b c] [d e} \kappa_{f]} \right)_\circ & = 0 \, ,
\end{align}
for any $1$-form $\kappa$ annihilating $K^\perp$, where $W_{a b c d}$ is the Weyl tensor of any metric in $\mbf{o}$.
\end{thm}

\begin{proof}
We first establish the equivalence of \eqref{item:lin_conn} and \eqref{item:non-sh-non-sh-cong} following the proof of \cite{Robinson1985}. Let $\nabla'$ be a torsion-free linear connection compatible with $\mbf{o}$ so that for any metric $g$ in $\mbf{o}$ and section $k$ of $K$, we have, with $\kappa = g(k, \cdot)$,
\begin{align*}
	\nabla'_a \kappa_b & = \alpha_a \kappa_b \, , \\
	\nabla'_a g_{b c} & = \beta_a g_{b c} + 2 \, \gamma_{a [b} \kappa_{c]}  \, ,
\end{align*}
for some tensor fields $\alpha_a$, $\beta_a$ and $\gamma_{a b}$. Then, since $\nabla$ is torsion-free,
\begin{align*}
	\kappa_{[a} (\d \kappa)_{b c]} & = \kappa_{[a} \nabla'_{b} \kappa_{c]} = 0 \, , \\
	\mathsterling_k g_{a b} & = k^c \nabla'_c g_{a b} + 2 \nabla'_{(a} \kappa_{b)} = (\beta_c k^c) g_{a b} + \left( k^c \gamma_{c (a} + \alpha_{(a} \right) \kappa_{b)} \, ,
\end{align*}
which shows that $\mc{K}$ is a non-twisting non-shearing congruence of null geodesics.

Conversely, suppose $\mc{K}$ is a non-twisting non-shearing congruence of null geodesics. Then we can find a metric $g$ in $\mbf{o}$ such that $\mc{K}$ is also non-expanding. The required linear connection is then the one found in Proposition \ref{prop:lin_conn_Kundt}.

For the final part of the proof, suppose that condition \eqref{item:non-sh-non-sh-cong} holds.
Then there is a metric $g$ in $\mbf{o}$ such that $\mc{K}$ is also non-expanding, i.e.\ $(\mc{M}, g, K)$ is a Kundt geometry. In particular, in the neighbourhood of any point, there exists coordinates $(u,v,x^i)$ such that $g$ takes the form \eqref{eq:Kundt_metric}, which we recast as
\begin{align*}
	g &  = 2 \,  \d u \d v + \ul{h}_{i j} \d x^i \d x^j  + 2 \, \d u  \left(  \ul{A}_i \d x^i + \ul{B} \d u \right) \, ,
\end{align*}
for some smooth functions $\ul{A}_i$ and $\ul{B}$ on $\ul{\mc{M}}$, and a family of symmetric tensors $\ul{h}_{i j}$ on $\ul{\mc{M}}$ smoothly parametrised by $u$, with the property that on restriction to each leaf of constant $u$, $\ul{h}_{i j}$ is a Riemannian metric. Note that $k = \frac{\partial}{\partial v}$ is a section of $K$, and $\kappa=g(k,\cdot) = \d u$.

When $n=2$, the metric $\ul{h}_{i j}$ is always conformally flat, and we may assume that $\ul{h}_{i j} = \e^{2 \varphi} \delta_{i j}$ where $\delta_{i j}$ is the standard Euclidean metric on $\R^2$, and $\varphi$ is a smooth function of $u$ and $x^i$. Now, define a new coordinate $\wt{v} = \e^{-2 \varphi} v$. Then
\begin{align*}
	g &  = \e^{2 \varphi} \left(2 \,  \d u \d \wt{v} + \delta_{i j} \d x^i \d x^j \right) + 2 \, \d u  \left(  \wt{\ul{A}}_i \d x^i + \wt{\ul{B}} \d u \right) \, ,
\end{align*}
for some functions $ \wt{\ul{A}}_i$ and $\wt{\ul{B}}$ of $u$ and $x^i$. This shows in particular that the Minkowski metric is in $\mbf{o}$, i.e., the $G$-structure is integrable.

In dimension $n>3$, on each leaf of constant $u$, the vanishing of the Weyl tensor $\ul{W}_{i j k \ell}$ of $\ul{h}_{i j}$ is a necessary and sufficient condition for $\ul{h}_{i j}$ to be conformally flat. Now, it is shown in \cite{Taghavi-Chabert2014} that condition \eqref{eq:W_cond} is equivalent to the equations
\begin{align*}
	k^a \delta^b_i \delta^c_j k^d W_{a b c d} & = 0 \, ,\\
	\left(k^a \delta^b_i \delta^c_j \delta^d_k W_{a b c d}\right)_\circ & = 0 \, , \\
	\left(\delta^a_i \delta^b_j \delta^c_k \delta^d_\ell W_{a b c d}\right)_\circ & = 0 \, .
\end{align*}
The first two conditions are the necessary curvature conditions for $\mc{K}$ to be a non-twisting non-shearing congruence of null geodesics -- see e.g. \cite{Podolsky2006a}. The last one is equivalent to $\ul{W}_{i j k \ell} = 0$  --- see for instance \cite{Taghavi-Chabert2014}.

Finally, we note that the condition \eqref{eq:W_cond} is independent of the choice of metric in $\mbf{o}$. Indeed, this condition refers only to the conformal structure on the screen bundle, and this is an invariant of the generalised optical geometry.

The converse is obvious.
\end{proof}

\begin{rem}
Proving an analogous statement for the second part of Theorem \ref{thm:integrable_optical} in dimension five, i.e.\ $n=3$, remains an open problem. The difficulty here is to find an appropriate curvature condition that is equivalent to the vanishing of the Cotton tensor of the (family of) metrics $\ul{h}_{i j}$.
\end{rem}

\begin{rem}
The connection given in Theorem \ref{thm:integrable_optical} is not unique.
\end{rem}

\begin{exa}[The Schwarzschild metric]
The generalised optical structure to which the Schwarzschild metric belongs is integrable in any dimensions. This follows from the following form  of the metric:
\begin{align*}
	g & = r^2 \left( 1 + \frac{1}{4} h_{k \ell} x^k x^\ell \right) h_{i j} \d x^i  \d x^j - 2 \d u \d r - 2 H \d r \d r \, .
\end{align*}
\end{exa}

\begin{exa}[The Myers-Perry metric]
The generalised optical structure to which the Kerr metric belongs is not integrable in any dimensions. Indeed, while one can cast the metric in Kerr-Schild form, the congruence of null geodesics is twisting. More specifically, let us write the Minkowski metric in standard coordinates $(t, x^\alpha , y^\alpha, z)_{\alpha=1,\ldots, m}$ in dimension $2m+2$,
\begin{align*}
	\eta & = - (\d t )^2 + \sum_{\alpha=1}^m \left( (\d x^\alpha)^2 +  (\d y^\alpha)^2 \right) + (\d z)^2  \, ,
\end{align*}
and let
\begin{align*}
	\kappa & = \d t + \sum_{\alpha=1}^m \frac{r(x^\alpha \d x^\alpha + y^\alpha \d y^\alpha)+a_\alpha(x^\alpha \d y^\alpha - y^\alpha \d x^\alpha)}{r^2 + a_\alpha^2} + \frac{z}{r} \d r \, ,
\end{align*}
and
\begin{align*}
	f & = \frac{M r^2}{1- \sum_{\alpha=1}^m  \frac{a_\alpha^2 \left( (x^\alpha)^2 + (y^\alpha)^2 \right)}{(r^2+a_\alpha^2)^2}} \frac{1}{\prod_{\alpha=1}^m (r^2 + a_\alpha^2)} \, .
\end{align*}
Here, the radial coordinate is defined by
\begin{align*}
	\sum_{\alpha=1}^m \frac{(x^\alpha)^2 + (y^\alpha)^2}{r^2 + a_\alpha^2} + \frac{z^2}{r^2} & = 1 \, .
\end{align*}
Then the Kerr-Myers-Perry metric in Kerr-Schild form is given by \cite{Myers1986}
\begin{align*}
	g & = \eta + f \kappa^2 \, .
\end{align*}
The odd-dimensional case is similar.
\end{exa}

	%
	%
	%
	
	
	\printbibliography
	
\end{document}